\documentclass{article}
\usepackage[title]{appendix}
\usepackage{graphicx}
\usepackage{amssymb}
\usepackage{graphicx,setspace}
\usepackage{psfrag}
\usepackage{amsfonts}
\expandafter\let\csname equation*\endcsname\relax
\usepackage{amsmath}
\usepackage{amssymb, amsthm}
\usepackage{mathrsfs}
\usepackage{epstopdf}
\usepackage{booktabs}
\usepackage{float}
\usepackage{cite}
\usepackage{lineno,hyperref}
\usepackage{color}
\usepackage{subfigure}
\usepackage{amsmath}
\usepackage{dsfont}
\usepackage{amssymb}
\usepackage{graphicx,color}
\usepackage{extarrows}
\usepackage{graphicx}
\usepackage{epstopdf}
\usepackage{booktabs}
\usepackage{array}
\usepackage{enumitem}
\usepackage{booktabs}
\usepackage{tabularx}
\usepackage{array}
\usepackage{ragged2e}
\usepackage{authblk}
\newcolumntype{L}[1]{>{\raggedright\arraybackslash}p{#1}}
\newcolumntype{Y}{>{\raggedright\arraybackslash}X}
\newcolumntype{C}[1]{>{\centering\arraybackslash}p{#1}}

\numberwithin{equation}{section}

\newtheorem{lemma}{Lemma}[section]
\newtheorem{remark}{Remark}[section]
\newtheorem{theorem}{Theorem}[section]
\newtheorem{proposition}{Proposition}[section]
\newtheorem{assumption}{Assumption}[section]

\title{Affine Option Pricing with Hawkes-Type Endogenous Jump Activity
}

\author[1]{Ziyang Fang }
\author[2]{Shenglan Yuan\footnote{Corresponding author: shenglanyuan@gbu.edu.cn}}

\affil[1]{\rm Business School, Southern University of Science and Technology, Shenzhen 518055, China}
\affil[2]{\rm Department of Mathematics, School of Sciences, Great Bay University, Dongguan 523000, China }

\begin{document}

\maketitle
\begin{abstract}
We develop a risk-neutral option-pricing model where the activity scale of an infinite-activity jump process is endogenously driven by the asset's own realized price jumps. Jump sizes are governed by a normalized asymmetric tempered-stable L\'evy shape, while a predictable activity scale controls the overall jump intensity and is normalized to coincide with the local jump-induced quadratic-variation rate. Endogenous feedback is introduced through the bounded excitation function \(g(y)=1-e^{-ay^2}\), so that small realized jumps excite future activity approximately in proportion to squared jump size while the total average excitation remains finite. We construct the coupled log-price and activity-state dynamics by state-dependent thinning of a Poisson random measure, prove pathwise existence and uniqueness, derive the mean-subcriticality condition, and obtain both the risk-neutral drift restriction and a sufficient true-martingale condition. The resulting two-dimensional state process admits an affine transform representation. We derive the associated generalized Riccati system and prove real-axis well-posedness with forward invariance of the relevant complex half-plane. European options are priced by a Fourier-cosine (COS) method, which requires only the real-axis transform, and are benchmarked against a damped Carr--Madan (CM) inversion. Numerical experiments illustrate the model-implied volatility surface and show how current activity shifts near-term volatility levels, while endogenous feedback affects the persistence of jump-induced skew across maturities.
\end{abstract}

\section{Introduction}

Discontinuities in asset prices are a central feature of financial return dynamics. Because of the jump-diffusion model of Merton~\cite{merton1976option}, jumps have been recognized as an essential ingredient for option pricing. They break the Black--Scholes replication argument, contribute to market incompleteness, and help generate short-maturity smiles and skews that are difficult to be reproduced with pure diffusions. A large literature has refined the cross-sectional shape of jump risk. In particular, Carr--Geman--Madan--Yor (CGMY)-type specifications of Carr, Geman, Madan, and Yor~\cite{carr2002fine} and the general tempered-stable framework of Rosi{\'n}ski~\cite{rosinski2007tempering} retain a stable-like singularity near the origin while exponentially tempering large positive and negative jumps. These models provide flexible infinite-activity return dynamics with asymmetric tails and finite moments. However, in many such specifications the overall activity of jump risk is either constant or driven by a stochastic factor that is exogenous to the realized jumps themselves.

A complementary empirical regularity is the clustering of jumps and extreme returns. Large moves tend to be followed by periods of elevated jump risk, and shocks may propagate across assets or markets. Self- and mutually exciting point processes, introduced by Hawkes~\cite{hawkes1971spectra} and connected to a branching representation by Hawkes and Oakes~\cite{hawkes1974cluster}, provide a natural framework for this phenomenon and have become important in financial modeling~\cite{bacry2015hawkes}. Two contributions are particularly relevant here. Errais, Giesecke, and Goldberg~\cite{errais2010affine} develop affine point processes for portfolio credit risk, where default events raise future default intensities and the affine structure yields tractable transform formulas. A\"{\i}t-Sahalia, Cacho-Diaz, and Laeven~\cite{ait2015modeling} model financial contagion through mutually exciting jump processes, in which jumps in one market region increase subsequent jump intensities both within and across regions.

These self-exciting models are event-based. The driving object is a finite-activity point process, and excitation is generated by the occurrence of observable events. This structure is natural for defaults, crashes, or other discrete events, but it does not directly match the infinite-activity description of returns used in tempered-stable models, where the number of jumps over a finite interval is infinite. This work bridges these two strands. We retain a normalized asymmetric tempered-stable jump-size shape, but we let its predictable activity scale be excited by realized price jumps. The feedback is not based on event counts. Instead, it is tied to realized jump variation through the bounded function
\[
g(y)=1-e^{-ay^2}.
\]
Since \(g(y)=ay^2+O(y^4)\) near the origin, small jumps increase future activity approximately in proportion to squared jump size. This is consistent with the normalization under which the activity scale equals the local jump-induced quadratic-variation rate. At the same time, boundedness of \(g\) ensures finite average excitation despite the infinite activity of the jump-size law. The resulting stability condition is therefore a first-moment subcriticality condition measuring jump-induced activity feedback relative to mean reversion, rather than a literal offspring-counting branching ratio.

The key distinction is that activity is endogenous to the asset's own realized price jumps, while the return jump component remains infinite activity. Since the objective is option pricing rather than statistical estimation, we work directly under a risk-neutral measure. This work makes five contributions. This work makes five contributions. First, we construct the coupled dynamics of the log price \(X_t\) and the activity state \(\lambda_t\) by state-dependent thinning of a Poisson random measure and prove pathwise existence, uniqueness, and finite-horizon integrability of the endogenous activity process. Second, we derive the first-moment dynamics of the activity and identify the corresponding mean-subcriticality condition, under which the activity has a finite stationary mean whenever a stationary distribution exists. Third, we obtain the risk-neutral drift restriction that makes the discounted stock price a local martingale and provide an explicit sufficient condition, based on an exponential moment of integrated activity, under which it is a true martingale. Fourth, we derive the conditional Fourier--Laplace transform of the joint state process \((X_t,\lambda_t)\). The associated generalized Riccati system is proved to be well-posed on the real Fourier axis, with the relevant complex half-plane forward-invariant. Fifth, we develop a Fourier-cosine pricing implementation, which relies only on this real-axis transform, and use a damped Carr--Madan inversion as an independent benchmark. The numerical study shows the model-implied volatility surface and decomposes the smile effects into current-activity and feedback-persistence channels.

Methodologically, the transform analysis builds on the affine-process framework of Duffie, Filipovi{\'c}, and Schachermayer~\cite{duffie2003affine} and the affine transform approach of Duffie, Pan, and Singleton~\cite{duffie2000transform}. It is closest in spirit to the affine self-exciting construction of Errais, Giesecke, and Goldberg~\cite{errais2010affine}, but differs in two important respects: the driving jump component has infinite activity, and feedback is generated by a bounded quadratic-variation-type function rather than by finite event arrivals. For this reason, the Riccati system is derived directly from the infinitesimal generator of the joint process.

The remainder of the paper is organized as follows. Section~\ref{sec:pre} introduces the normalized tempered-stable jump-size shape and the activity-scaled compensator. Section~\ref{sec:construction_transform} constructs the Hawkes-type endogenous-activity model under the risk-neutral measure and establishes well-posedness, first-moment subcriticality, the risk-neutral drift restriction, the affine transform, and true-martingale conditions. Section~\ref{sec:fourier_pricing} develops the COS and Carr--Madan pricing formulas. Section~\ref{sec:numerical} reports the numerical implementation, pricing diagnostics, and implied-volatility analysis. Section~\ref{sec:future_challenge} concludes and discusses future challenges.Proofs are collected in the appendices.

\section{Preliminaries}\label{sec:pre}
This section introduces the basic modeling framework. The construction is based on a separation between two components of jump risk. The first component is a normalized L\'evy shape, denoted by \(\bar\nu(dy;\vartheta)\), which determines the relative distribution of jump sizes. It controls tail asymmetry, tail tempering, small-jump activity, and the cross-sectional shape of jump risk. The second component is a predictable activity scale, denoted by \(\lambda_{t-}\), which determines the time variation in the overall intensity of this normalized jump law.

\subsection{Normalized Tempered-Stable L\'evy Shape}
Let
\[
\vartheta=(p,M,G,\alpha),
\]
with \(p\in(0,1)\), \(M,G>0\) and \(\alpha\in(0,2)\). We define the normalized asymmetric tempered-stable L\'evy shape by
\begin{equation}\label{eq:Levy_density}
\bar\nu(dy;\vartheta)=\left[p\frac{M^{2-\alpha}}{\Gamma(2-\alpha)}\frac{e^{-My}}{y^{1+\alpha}}\mathbf 1_{\{y>0\}}+(1-p)\frac{G^{2-\alpha}}{\Gamma(2-\alpha)}\frac{e^{-G|y|}}{|y|^{1+\alpha}}\mathbf 1_{\{y<0\}}\right]dy.
\end{equation}
This specification is a normalized version of the asymmetric tempered-stable L\'evy density used in CGMY-type return models \cite{carr2002fine}. It retains the stable-like singularity near zero while exponentially tempering large positive and negative jumps, in line with the general tempered-stable construction \cite{rosinski2007tempering}.

The following result collects the basic admissibility, variation, and moment properties of the normalized shape.
\begin{theorem}\label{thm:normalized-shape-admissibility}
For every \(\vartheta\), the measure \(\bar\nu(dy;\vartheta)\) is a valid L\'evy measure. Moreover:
\begin{enumerate}[label=(\roman*)]
\item \(\bar\nu(|y|<1)=\infty\). Hence the jump component has infinite activity.
\item The small jumps have finite path variation if and only if
\[
\alpha<1.
\]
\item For every integer \(n\ge2\), the normalized jump-size moment
\[
c_n(\vartheta):=\int_{\mathbb R_0}y^n\bar\nu(dy;\vartheta)
\]
is finite and is given by
\[
c_n(\vartheta)=\frac{\Gamma(n-\alpha)}{\Gamma(2-\alpha)}\left[pM^{2-n}+(-1)^n(1-p)G^{2-n}\right].
\]
In particular, \(c_2(\vartheta)=1\),
\[
\int_{y>0}y^2\bar\nu(dy;\vartheta)=p,
\qquad\text{and}\qquad
\int_{y<0}y^2\bar\nu(dy;\vartheta)=1-p.
\]
\end{enumerate}
\end{theorem}

\begin{proof}
See Appendix~\ref{app:Levy_shape_results}.
\end{proof}

\subsection{Activity-Scaled Jump Compensator}
\label{subsec:activity_scaled_jump_compensator}
Now we introduce the activity scale under a generic probability measure. Let
\[
(\Omega,\mathcal F,\mathbb F,\cdot),
\qquad
\mathbb F=(\mathcal F_t)_{t\geq0}
\]
be a filtered probability space satisfying the usual conditions. Here ``\(\cdot\)'' denotes either the physical measure \(P\) or the risk-neutral measure \(Q\). Let
\[
\lambda^\cdot=(\lambda_t^\cdot)_{t\geq0}
\]
be a nonnegative, \(\mathbb F\)-adapted c\`adl\`ag activity process satisfying
\[
\int_0^T\lambda_{s}^\cdot\,ds<\infty
\qquad
\text{a.s. for every }T<\infty.
\]
A jump compensator is said to be activity-scaled if it admits the factorized form
\begin{equation}
\label{eq:activity_scaled_compensator_generic}
\nu_t^\cdot(dy)\,dt=\lambda_{t-}^\cdot\bar\nu(dy;\vartheta^\cdot)\,dt.
\end{equation}
By Theorem~\ref{thm:normalized-shape-admissibility}, we have
\[
\int_{\mathbb R_0}y^2\nu_t^{\cdot}(dy)=\lambda_{t-}^{\cdot}.
\]
Hence \(\lambda_{t-}^{\cdot}\) is the local jump-induced quadratic-variation rate.

\begin{remark}[Predictability convention] The predictable activity rate entering the compensator is the left-limit process \(\lambda_{t-}^\cdot\). This convention ensures that jump arrivals over \([t,t+dt)\) are governed by the information available immediately before time \(t\), rather than by the post-jump value of the activity process. Moreover, since \(\lambda^\cdot\) is c\`adl\`ag, \(\lambda_s^\cdot\) and \(\lambda_{s-}^\cdot\) differ only at jump times, and hence have the same Lebesgue integral over finite horizons. \end{remark}

The notation in this subsection is deliberately measure-generic. In models with jumps, dynamic trading in the stock and the money-market account does not in general eliminate jump risk. This point is already apparent in \cite{merton1976option}, where discontinuous stock returns break the Black--Scholes replication argument, and jump models are typically incomplete; see \cite{cont2003financial,eberlein1997range}. Hence absence of arbitrage does not determine a unique equivalent martingale measure, and we do not derive a unique \(\mathbb P\)-to-\(\mathbb Q\) transformation from first principles. Instead, the purpose of this subsection is to specify a compensator class that can be imposed under either measure. Under the physical measure, the factorization in \eqref{eq:activity_scaled_compensator_generic} may be read as a statistical decomposition of jump-size shape and jump-activity scale. Under a risk-neutral measure, the same factorized form is used as a pricing specification, possibly with different parameters \(\vartheta^Q\) and a different activity process \(\lambda^Q\), as in the CGMY literature where statistical and risk-neutral parameters are typically distinguished \cite{carr2002fine}. In the next section, we work directly under a risk-neutral measure \(\mathbb Q\), preserving the shape--activity compensator structure while imposing the martingale drift restriction separately.

\section{Risk-Neutral Model Construction and Affine Transform}\label{sec:construction_transform}
Throughout this section, all processes, compensators and expectations are understood under the risk-neutral measure \(\mathbb Q\). For notational simplicity, we suppress the superscript \(\mathbb Q\).

\begin{assumption}[Primitive driving noises]
\label{ass:primitive_noises}
The filtered probability space satisfies the usual conditions and carries a Brownian motion \(W\) and an independent Poisson random measure
\[
\mathcal N(dt,dy,d\gamma)
\]
on
\[
[0,\infty)\times\mathbb R_0\times[0,\infty),
\qquad\mathbb R_0:=\mathbb R\setminus\{0\},
\]
with intensity measure
\[
dt\,\bar\nu(dy;\vartheta)\,d\gamma .
\]
\end{assumption}

\subsection{Modeling Framework}
Let
\[
X_t=\log S_t
\]
denote the log-price process. The jump measure of \(X\) is constructed by thinning the primitive Poisson random measure \(\mathcal N\). More precisely, for a nonnegative predictable activity process \(\lambda\), define
\begin{equation}\label{eq:jump_measure}
\mu(dt,dy):=\int_0^\infty\mathbf 1_{\{\gamma\le \lambda_{t-}\}}\,\mathcal N(dt,dy,d\gamma).
\end{equation}
The pair \((\lambda,\mu)\) is not assumed to exist a priori. It is constructed as the solution of the coupled self-exciting system below.

The endogenous activity process is specified by
\begin{equation}
\label{eq:endogenous_activity}
d\lambda_t=\kappa(\bar\lambda-\lambda_t)\,dt+\eta\int_{\mathbb R_0}g(y)\,\mu(dt,dy),
\end{equation}
where
\[
\kappa>0,\qquad \bar\lambda>0,\qquad \eta\ge 0,
\]
and the initial activity level \(\lambda_0>0\) is deterministic. The excitation function is chosen as
\begin{equation}\label{eq:g}
g(y)=1-e^{-ay^2},\qquad a>0.
\end{equation}
Since
\[
0\le g(y)\le ay^2,
\]
Theorem~\ref{thm:normalized-shape-admissibility} implies
\begin{equation}\label{eq:bar_g}
\bar g:=\int_{\mathbb R_0}g(y)\,\bar\nu(dy;\vartheta)<\infty .
\end{equation}
As \(y\to 0\), we have
\begin{equation}\label{eq:origin_behavior_g}
g(y)=ay^2+O(y^4).
\end{equation}
Hence small realized jumps excite future activity approximately in proportion to squared jump size. This is consistent with the interpretation of \(\lambda_{t-}\) as the local jump-induced quadratic-variation rate.

\begin{proposition}[Well-posedness of the endogenous activity construction]\label{prop:well_posed_activity}
Under Assumption~\ref{ass:primitive_noises}, the coupled system consisting of \eqref{eq:jump_measure} and \eqref{eq:endogenous_activity} admits a pathwise unique solution \((\lambda,\mu)\), where \(\lambda\) is a nonnegative c\`adl\`ag adapted process and \(\mu\) is the associated thinned jump measure. Moreover, for every finite horizon \(T<\infty\),
\[
\int_0^T\lambda_s\,ds<\infty
\qquad\text{a.s.},\qquad\text{and}\qquad\mathbb E\!\left[\int_0^T\lambda_s\,ds\right]<\infty.
\]
\end{proposition}

\begin{proof}
See Appendix~\ref{app:well_posed_activity}.
\end{proof}

By Proposition~\ref{prop:well_posed_activity}, the state-dependent thinning of the primitive Poisson random measure \cite{kingman1992poisson,daley2003introduction} and the endogenous activity equation generate a self-consistent pair \((\lambda,\mu)\). For the thinned jump measure \(\mu\), the predictable compensator is
\begin{equation}
\label{eq:activity_scaled_compensator}
\nu_t(dy)\,dt=\lambda_{t-}\bar\nu(dy;\vartheta)\,dt .
\end{equation}
Indeed, the primitive Poisson random measure has intensity \(dt\,\bar\nu(dy;\vartheta)d\gamma\), and the thinning rule accepts precisely those candidate jumps with \(\gamma\le\lambda_{t-}\). Integrating over the auxiliary thinning coordinate therefore gives
\[
\int_0^\infty\mathbf 1_{\{\gamma\le\lambda_{t-}\}}\,d\gamma=\lambda_{t-}.
\]
Thus, the thinning construction together with the endogenous activity equation \eqref{eq:endogenous_activity} provides a concrete realization of the activity-scaled compensator class introduced in Section~\ref{subsec:activity_scaled_jump_compensator}. In particular, once the primitive Poisson random measure and the activity equation are fixed, the factorized compensator in \eqref{eq:activity_scaled_compensator} follows from the construction rather than being imposed as an additional assumption.

We write the compensated jump measure as
\[
\widetilde\mu(dt,dy)=\mu(dt,dy)-\lambda_{t-}\bar\nu(dy;\vartheta)\,dt .
\]
Therefore, the log-price dynamics are then specified by
\begin{equation}
\label{eq:log_price}
dX_t=b_t\,dt+\sigma\,dW_t+\int_{|y|<1} y\,\widetilde\mu(dt,dy)+\int_{|y|\ge 1} y\,\mu(dt,dy).
\end{equation}

The preceding construction establishes that the endogenous activity scale is well defined on finite horizons and that the compensator in \eqref{eq:activity_scaled_compensator} is admissible. We now use this finite-horizon admissibility to derive the first-moment dynamics of \(\lambda\) and to identify the corresponding mean-subcriticality condition.

\begin{proposition}\label{prop:mean_subcriticality}
The first moment
\[
m(t):=\mathbb E[\lambda_t]
\]
satisfies
\[
m'(t)=\kappa\bar\lambda-(\kappa-\eta\bar g)m(t),
\qquad m(0)=\lambda_0.
\]
Consequently, the natural mean-subcriticality condition is
\begin{equation}\label{eq:subcriticality_condition}
\kappa>\eta\bar g.
\end{equation}
Under this condition,
\[
\lim_{t\to\infty}m(t)=\frac{\kappa\bar\lambda}{\kappa-\eta\bar g}.
\]
In particular, if a stationary distribution with finite first moment exists, its mean must be
\[
\mathbb E[\lambda_\infty]=\frac{\kappa\bar\lambda}{\kappa-\eta\bar g}.
\]
\end{proposition}

\begin{proof}
By Proposition~\ref{prop:well_posed_activity}, the finite-horizon integrability needed for the compensation formula is satisfied. Integrating activity dynamics in \eqref{eq:endogenous_activity} and taking expectations gives
\[
m(t)=m(0)+\int_0^t\left[\kappa\bar\lambda-\kappa m(s)\right]ds+\eta\,\mathbb E\left[\int_0^t\int_{\mathbb R_0}g(y)\,\mu(ds,dy)\right].
\]
Using the compensator in \eqref{eq:activity_scaled_compensator} and the definition of \(\bar g\) in \eqref{eq:bar_g}, we obtain
\[
\mathbb E\left[\int_0^t\int_{\mathbb R_0}g(y)\,\mu(ds,dy)\right]=\bar g\int_0^t m(s)\,ds,
\]
where \(\lambda_s=\lambda_{s-}\) for Lebesgue-a.e. \(s\in[0,t]\). Then
\[
m(t)=m(0)+\int_0^t\left[\kappa\bar\lambda-(\kappa-\eta\bar g)m(s)\right]ds,
\]
and therefore
\[
m'(t)=\kappa\bar\lambda-(\kappa-\eta\bar g)m(t).
\]
Solving this linear equation yields, for \(\kappa\neq\eta\bar g\),
\begin{equation}\label{eq:first_moment_m}
m(t)=e^{-(\kappa-\eta\bar g)t}m(0)+\frac{\kappa\bar\lambda}{\kappa-\eta\bar g}\left(1-e^{-(\kappa-\eta\bar g)t}\right).
\end{equation}
Thus \(m(t)\) converges to \(\kappa\bar\lambda/(\kappa-\eta\bar g)\) when \(\kappa>\eta\bar g\). If a stationary distribution with finite first moment exists, stationarity forces \(m'(t)=0\), and the same expression gives the only possible stationary mean.
\end{proof}

The same first-moment formula applies to any nonnegative random initial activity with \(\mathbb E[\lambda_0]<\infty\), after replacing \(m(0)\) by \(\mathbb E[\lambda_0]\). The next proposition strengthens the mean-stability result under \(\kappa>\eta\bar g\) to distributional stability and identifies \(\kappa\bar\lambda/(\kappa-\eta\bar g)\) as the stationary activity mean.

\begin{proposition}[Stationarity and geometric ergodicity of the activity process]\label{prop:stationary_activity}
Suppose the standing assumptions hold and the mean-subcriticality condition in \eqref{eq:subcriticality_condition} is satisfied. Then the activity process \((\lambda_t)_{t\ge0}\) admits a unique invariant distribution \(\pi_\lambda\), supported on \([\bar\lambda,\infty)\). Moreover, there exist constants \(\mathfrak M<\infty\) and \(\mathfrak C>0\) such that, for every \(\lambda_0\ge0\),
\[
\left\|\mathcal L(\lambda_t\mid \lambda_0)-\pi_\lambda\right\|_{\mathrm{TV}}\le\mathfrak M(1+\lambda_0)e^{-\mathfrak C t},
\qquad t\ge0.
\]
Here \(\|\cdot\|_{\mathrm{TV}}\) denotes total variation distance. The invariant distribution has finite first moment and
\[
\int_0^\infty \lambda\,\pi_\lambda(d\lambda)=\frac{\kappa\bar\lambda}{\kappa-\eta\bar g}.
\]
This coincides with the limiting value of the finite-horizon mean \(m(t)=\mathbb E[\lambda_t\mid\lambda_0]\) in \eqref{eq:first_moment_m}.
\end{proposition}

\begin{proof}
See Appendix~\ref{app:proof_stationarity_geometric_ergodicity}.
\end{proof}

When \(\kappa\le \eta\bar g\), the first-moment equation does not converge to a finite positive level: the mean grows linearly at equality and diverges when \(\kappa<\eta\bar g\). Thus the stability condition can equivalently be written as
\[
\mathcal B:=\frac{\eta\bar g}{\kappa}<1.
\]
We interpret \(\mathcal B\) as a model-specific analogue of the branching ratio in linear Hawkes models, where stationarity requires the integrated excitation to be below one \cite{hawkes1971spectra,hawkes1974cluster,bacry2015hawkes}. In the present infinite-activity setting, this interpretation is structural rather than genealogical: \(\mathcal B\) measures average jump-induced activity feedback relative to mean reversion, not a literal mean number of offspring events. Accordingly, the bounded excitation kernel \(g\) in \eqref{eq:g} should be read as a reduced-form response to realized jump variation, rather than as a separate micro-foundation. Throughout the sequel, we impose \(\mathcal B<1\), or equivalently \(\kappa>\eta\bar g\).

\subsection{Risk-Neutral Drift Restriction}
Let \(r\) denote the risk-free rate. For risk-neutral pricing, we impose the positive-tail condition \(M>1\), which ensures the finiteness of the positive exponential moment of the tempered-stable positive tail. This is consistent with the standard exponential-moment criterion for L\'evy processes \cite[Proposition~3.14]{cont2003financial}; in tempered-stable stock models, the moment domain is determined by the tempering parameters \cite{kuchler2014exponential}.

\begin{lemma}\label{lem:jump_martingale_correction_finite}
Suppose that \(M>1\). Then
\begin{equation}\label{eq:chi_J}
\chi_J := \int_{\mathbb R_0} \left( e^y-1-y\mathbf 1_{\{|y|<1\}} \right) \bar\nu(dy;\vartheta)
\end{equation}
is finite.
\end{lemma}

\begin{proof}
We first consider the small-jump region. Since \(\bar\nu(dy;\vartheta)\) has finite second moment from Theorem~\ref{thm:normalized-shape-admissibility}, we have
\[
\int_{|y|<1} \left|e^y-1-y\right| \bar\nu(dy;\vartheta)<\infty.
\]
It remains to control the large-jump region. Using the positive part of \eqref{eq:Levy_density},
\[
\int_1^\infty e^y\bar\nu(dy;\vartheta)=p\frac{M^{2-\alpha}}{\Gamma(2-\alpha)}\int_1^\infty e^{-(M-1)y}y^{-1-\alpha}\,dy<\infty
\]
when \(M>1\). On the negative tail, since \(\bar\nu\) is a L\'evy measure by Theorem~\ref{thm:normalized-shape-admissibility},
\[
\int_{-\infty}^{-1}|e^y-1|\bar\nu(dy;\vartheta)\leq\bar\nu\left((-\infty,-1];\vartheta\right)<\infty.
\]
Hence the large-jump contribution is finite. Combining the small- and large-jump regions gives \(\chi_J<\infty\).
\end{proof}

\begin{proposition}[Risk-neutral drift restriction]
\label{prop:risk_neutral_drift_restriction}
Suppose that \(M>1\), so that \(\chi_J\) is finite. If the log-price drift in \eqref{eq:log_price} is chosen as
\begin{equation}\label{eq:risk_neutral_drift_restriction}
b_t=r-\frac12\sigma^2-\lambda_{t-}\chi_J,
\end{equation}
then the discounted stock price
\[
e^{-rt}S_t
\]
is a local martingale.
\end{proposition}

\begin{proof}
Applying It\^o's formula for the jump semimartingale in \eqref{eq:log_price} gives
\[
\frac{dS_t}{S_{t-}}=\left(b_t+\frac12\sigma^2+\lambda_{t-}\int_{\mathbb R_0}\left(e^y-1-y\mathbf 1_{\{|y|<1\}}\right)\bar\nu(dy;\vartheta)\right)dt+dM_t,
\]
where \(M_t\) is a local martingale collecting the Brownian martingale part and
the compensated jump part. Hence
\[
dS_t=rS_{t-}dt+S_{t-}dM_t.
\]
Equivalently,
\[
d(e^{-rt}S_t)=e^{-rt}S_{t-}dM_t,
\]
so \(e^{-rt}S_t\) is a local martingale.
\end{proof}

\subsection{Affine Transform Representation}\label{subsec:affine_transform_representation}
We derive the conditional Fourier--Laplace transform of the joint state process \((X_t,\lambda_t)\). The log-price jumps and the activity state are both coupled through \(\mu(dt,dy)\). The predictable activity level \(\lambda_{t-}\) scales the compensator of the jump measure, while a realized jump of size \(y\) increases activity by
\[
\Delta\lambda_t=\eta g(y).
\]
This feedback from realized jumps to future activity gives the model its self-exciting structure.

The transform calculation follows the affine-transform methodology of \cite{duffie2000transform} and is consistent with the affine-process framework of \cite{duffie2003affine}. A related financial application appears in \cite{errais2010affine}, where affine point processes are used to model self-exciting event arrivals in portfolio credit risk. Our model shares the affine self-exciting structure, but differs because the driving jumps have infinite activity and the feedback is generated by the bounded quadratic-variation function \(g\). We therefore derive the Riccati system directly from the infinitesimal generator rather than invoking a finite-activity point-process transform formula.

Let
\begin{equation}\label{eq:tau}
\tau=T-t.
\end{equation}
For \(u\in\mathbb R\) and \(v\in\mathbb C\) in a domain where the transform is finite,
\begin{equation}\label{eq:exponential-affine_candidate_defi}
F(\tau,x,\ell;u,v):=\mathbb E\left[\exp\left(iuX_T+v\lambda_T\right)\,\middle|\,X_t=x,\lambda_t=\ell\right].
\end{equation}
We consider the exponential-affine candidate representation
\begin{equation}\label{eq:exponential-affine_candidate_representation}
F(\tau,x,\ell;u,v)=\exp\left(iux+\phi(\tau;u,v)+\psi(\tau;u,v)\ell\right),
\end{equation}
where
\[
\phi(0;u,v)=0,
\qquad
\psi(0;u,v)=v.
\]
For each fixed value of \(\tau\),
\[
f_\tau(x,\ell) := F(\tau,x,\ell;u,v) = \exp\left( iux+\phi(\tau;u,v)+\psi(\tau;u,v)\ell \right).
\]
The generator \(\mathcal A\) acts only on the state variables \((x,\ell)\). For a state-space test function \(h\) in the domain of the generator, the infinitesimal generator of the joint Markov process \((X_t,\lambda_t)\) is
\begin{equation}\label{eq:joint_generator}
\begin{aligned}
\mathcal A h(x,\ell)
&=\left(r-\frac12\sigma^2-\ell\chi_J\right)\partial_x h(x,\ell)+\kappa(\bar\lambda-\ell)\partial_\ell h(x,\ell)+\frac12\sigma^2\partial_{xx}h(x,\ell)
\\
&+\ell\int_{\mathbb R_0}\left[h\bigl(x+y,\ell+\eta g(y)\bigr)-h(x,\ell)-y\mathbf 1_{\{|y|<1\}}\partial_x h(x,\ell)\right]\bar\nu(dy;\vartheta).
\end{aligned}
\end{equation}

\begin{remark}[Raw and compensated activity excitation]
From \eqref{eq:bar_g} we know \(\bar g <\infty\), therefore, the activity equation \eqref{eq:endogenous_activity} can equivalently be written using the compensated jump measure:
\[
d\lambda_t=\left(\,\kappa(\bar\lambda-\lambda_t)+\eta\lambda_{t-}\bar g\,\right)dt+\eta\int_{\mathbb R_0}g(y)\,\widetilde\mu(dt,dy).
\]
Under this equivalent representation, the generator contains the drift correction \(\eta\ell\bar g\,\partial_\ell h(x,\ell)\) and the jump integral contains the additional truncation term \(-\eta g(y)\partial_\ell h(x,\ell)\). These two terms cancel after integration with \(\ell\bar\nu(dy;\vartheta)\), yielding the generator form in \eqref{eq:joint_generator}.
\end{remark}
Evaluating the generator at the exponential-affine test function \(h=f_\tau\), we obtain
\begin{equation}\label{eq:division_generator}
\frac{\mathcal A f_\tau(x,\ell)}{f_\tau(x,\ell)}=\mathcal K_0\bigl(u,\psi(\tau;u,v)\bigr)+\ell\,\mathcal K_1\bigl(u,\psi(\tau;u,v)\bigr),
\end{equation}
where for a generic complex argument \(z\),
\begin{equation}
\mathcal K_0(u,z)=iu\left(r-\frac12\sigma^2\right)\frac12\sigma^2u^2+\kappa\bar\lambda\,z
\end{equation}
and
\begin{equation}\label{eq:K_1}
\mathcal K_1(u,z)=-\kappa z-iu\chi_J+\int_{\mathbb R_0}\left[\exp\left(iuy+\eta z g(y)\right)-1-iuy\mathbf 1_{\{|y|<1\}}\right]\bar\nu(dy;\vartheta).
\end{equation}

Since the joint state process is time homogeneous, the conditional transform satisfies the backward equation
\[
\partial_\tau F(\tau,x,\ell;u,v)=\mathcal A\bigl(F(\tau,\cdot,\cdot;u,v)\bigr)(x,\ell),
\qquad F(0,x,\ell;u,v)=e^{iux+v\ell}.
\]
On the other hand, by \eqref{eq:exponential-affine_candidate_representation} we have
\begin{equation}\label{eq:division_transform}
\frac{\partial_\tau F(\tau,x,\ell;u,v)}{F(\tau,x,\ell;u,v)}=\phi'(\tau;u,v)+\ell\,\psi'(\tau;u,v).
\end{equation}
Matching \eqref{eq:division_generator} and \eqref{eq:division_transform} yields
\[
\left\{\begin{aligned}
\phi'(\tau;u,v)
&=\mathcal K_0\bigl(u,\psi(\tau;u,v)\bigr),
\\
\psi'(\tau;u,v)
&=\mathcal K_1\bigl(u,\psi(\tau;u,v)\bigr).
\end{aligned}
\right.
\]
Hence the generalized Riccati system is
\begin{equation}\label{eq:generalized_Riccati_system}
\left\{
\begin{aligned}
&\phi'(\tau;u,v)=iu\left(r-\frac12\sigma^2\right)-\frac12\sigma^2u^2+\kappa\bar\lambda\,\psi(\tau;u,v),
\\[0.5em]
&\psi'(\tau;u,v)=-\kappa\psi(\tau;u,v)-iu\chi_J
\\
&\qquad+\int_{\mathbb R_0}\Bigl[\exp\bigl(iuy+\eta\psi(\tau;u,v)g(y)\bigr)-1-iuy\mathbf 1_{\{|y|<1\}}\Bigr]\bar\nu(dy;\vartheta),\\[0.5em]
&\phi(0;u,v)=0,\qquad\psi(0;u,v)=v.
\end{aligned}
\right.
\end{equation}
The following theorem establishes real-axis well-posedness of the Riccati system and verifies the corresponding affine transform identity.

\begin{theorem}[Real-axis Riccati system and affine transform]\label{thm:real_axis_riccati}
Suppose that the conditions of Proposition~\ref{prop:risk_neutral_drift_restriction} hold, and define
\[
\mathbb C_-:=\{z\in\mathbb C:\operatorname{Re}z\le0\}.
\]
For every \(u\in\mathbb R\), \(v\in\mathbb C_-\), and finite horizon \(\bar T<\infty\), the Riccati equation
\begin{equation}\label{eq:riccati_psi}
\psi'(\tau;u,v)=\mathcal K_1\bigl(u,\psi(\tau;u,v)\bigr),
\qquad\psi(0;u,v)=v
\end{equation}
has a unique solution on \([0,\bar T]\), and this solution remains in the closed left half-plane:
\[
\psi(\tau;u,v)\in\mathbb C_-,
\qquad 0\le \tau\le \bar T.
\]
With
\[
\phi(\tau;u,v):=\int_0^\tau\mathcal K_0\bigl(u,\psi(s;u,v)\bigr),ds,
\]
the pair \((\phi,\psi)\) is well defined on every finite horizon and solves the Riccati system. Consequently, for \(0\leq t\leq T<\infty\), with \(\tau=T-t\),
\[
\mathbb E_t\left[\exp\left(iuX_T+v\lambda_T\right)\right]=\exp\left(iuX_t+\phi(\tau;u,v)+\psi(\tau;u,v)\lambda_t\right).
\]
\end{theorem}

\begin{proof}
See Appendix~\ref{app:proof_real_axis_riccati}.
\end{proof}

\begin{remark}[Stock-price moment consistency]
\label{rem:stock_price_moment_consistency}
The Riccati system in \eqref{eq:generalized_Riccati_system} also provides a useful consistency check at the formal moment \((u,v)=(-i,0)\), although this point lies outside the real Fourier axis covered by Theorem~\ref{thm:real_axis_riccati}. Under the risk-neutral drift restriction in Proposition~\ref{prop:risk_neutral_drift_restriction}, the Riccati vector field satisfies
\[
\mathcal K_1(-i,0)=0,
\qquad
\mathcal K_0(-i,0)=r.
\]
Hence the formal Riccati solution is
\[
\psi(\tau;-i,0)\equiv0,
\qquad
\phi(\tau;-i,0)=r\tau.
\]
\end{remark}

The affine transform provides a tractable representation of the real-axis conditional characteristic function. The stock-price moment and true-martingale property require separate exponential-integrability conditions discussed below.

\begin{theorem}[A sufficient condition for the true martingale]
\label{thm:true_martingale_sufficient_condition}
Suppose that Proposition~\ref{prop:risk_neutral_drift_restriction} holds. Define
\[
H_J:=\int_{\mathbb R_0}\left(y e^y-e^y+1\right)\bar\nu(dy;\vartheta).
\]
Then \(H_J<\infty\). If for every finite horizon \(T<\infty\),
\[
\mathbb E\left[\exp\left(H_J\int_0^T\lambda_s\,ds\right)\right]<\infty,
\]
then the discounted stock price \(e^{-rt}S_t\) is a true martingale on
\([0,T]\).
\end{theorem}

\begin{proof}
See Appendix~\ref{app:true_martingale_sufficient_condition}.
\end{proof}

\begin{proposition}[Integrated-activity exponential moment bound]\label{prop:integrated_activity_exponential_moment}
Let \(c\ge0\). Suppose that there exists a constant \(R>0\) such that
\[
c-\kappa R+\int_{\mathbb R_0}\left(e^{\eta R g(y)}-1\right)\bar\nu(dy;\vartheta)\le 0.
\]
Then for every finite horizon \(T<\infty\),
\[
\mathbb E\left[\exp\left(c\int_0^T\lambda_s\,ds\right)\right]<\infty.
\]
Consequently, by Theorem~\ref{thm:true_martingale_sufficient_condition}, \(e^{-rt}S_t\) is a true martingale on every finite horizon whenever the displayed condition holds with \(c=H_J\).
\end{proposition}

\begin{proof}
The case \(c=0\) is trivial. For \(c>0\), we define
\begin{equation}\label{eq:mathfrak_h}
\mathfrak h(x) := c-\kappa x + \int_{\mathbb R_0} \left( e^{\eta xg(y)}-1 \right)\bar\nu(dy;\vartheta), \qquad x\ge0.
\end{equation}
For any \(R^*>0\) and \(0\le x_1,x_2\le R^*\),
\[
\left|e^{\eta x_1 g(y)}-e^{\eta x_2 g(y)}\right|\le\eta e^{\eta R^*} g(y)|x_1-x_2|.
\]
Therefore, we have
\[
|\mathfrak h(x_1)-\mathfrak h(x_2)|\le\bigl(\kappa+\eta e^{\eta R^*}\bar g\bigr)|x_1-x_2|,
\]
which means \(\mathfrak h\) is locally Lipschitz on \([0,\infty)\).

Consider the scalar Riccati equation \[ B'(q)=\mathfrak h(B(q)), \qquad B(0)=0 . \] Since \(\mathfrak h(0)=c\ge0\), the solution cannot cross below \(0\). Since the assumed condition gives \(\mathfrak h(R)\le0\), the solution cannot cross above \(R\). By the standard one-dimensional invariance argument for ordinary differential equations (ODEs) \cite[Chapter~1]{coddington1955theory}, the interval \([0,R]\) is forward invariant. Hence \[ 0\le B(q)\le R, \qquad q\ge0 . \] Define \[ A(q):=\kappa\bar\lambda\int_0^q B(\varpi)\,d\varpi, \qquad q\ge0 . \] Then \(A(q)<\infty\) for every finite \(q\). We now verify the exponential-moment bound directly, following the affine-transform construction of \cite{duffie2000transform,duffie2003affine}. For \(t\in[0,T]\), write \(q=T-t\) and set \[ \mathfrak Z_t := \exp\!\left( c\int_0^t\lambda_s\,ds + A(q) + B(q)\lambda_t \right). \] The process \((\mathfrak Z_t)_{0\le t\le T}\) is nonnegative and adapted. Applying It\^o's formula to \(\mathfrak Z_t\), and compensating the jump term by \(\lambda_{t-}\bar\nu(dy;\vartheta)\,dt\), the predictable finite-variation part has density \[ \mathfrak Z_{t-}\,\mathcal G(q,\lambda_{t-}), \] where \[ \mathcal G(q,\lambda) = \Bigg[ c - B'(q) - \kappa B(q) + \int_{\mathbb R_0} \bigl(e^{\eta B(q)g(y)}-1\bigr) \bar\nu(dy;\vartheta) \Bigg]\lambda + \bigl[ \kappa\bar\lambda B(q)-A'(q) \bigr]. \] The coefficient of \(\lambda\) is \[ \mathfrak h(B(q))-B'(q)=0, \] by the Riccati equation, and the constant term is \[ \kappa\bar\lambda B(q)-A'(q)=0, \] by the definition of \(A\). Hence \(\mathcal G\equiv0\). It remains only to note that the jump compensator is finite. Since \(0\le B(q)\le R\) and \(0\le g(y)\le1\), the inequality \(e^x-1\le xe^x\) for \(x\ge0\) gives \[ \int_{\mathbb R_0} \bigl(e^{\eta B(q)g(y)}-1\bigr) \bar\nu(dy;\vartheta) \le \eta R e^{\eta R} \int_{\mathbb R_0}g(y)\,\bar\nu(dy;\vartheta) = \eta R e^{\eta R}\bar g <\infty . \] Therefore \(\mathfrak Z\) is a nonnegative local martingale, hence a supermartingale. Since \(A(0)=B(0)=0\), \[ \mathfrak Z_T = \exp\!\left(c\int_0^T\lambda_s\,ds\right), \qquad \mathfrak Z_0 = \exp\!\bigl(A(T)+B(T)\lambda_0\bigr). \] Consequently, \[ \mathbb E\!\left[ \exp\!\left( c\int_0^T\lambda_s\,ds \right) \right] = \mathbb E[\mathfrak Z_T] \le \mathfrak Z_0 = \exp\!\bigl(A(T)+B(T)\lambda_0\bigr) <\infty .
\]
This proves the claim.
\end{proof}

The condition in Proposition~\ref{prop:integrated_activity_exponential_moment} is stronger than the mean-subcriticality condition \(\kappa>\eta\bar g\). The latter controls the first moment of the activity process and ensures that the mean equation has a finite stationary level. By contrast, Proposition~\ref{prop:integrated_activity_exponential_moment} controls a positive exponential moment of the integrated activity. Indeed, for \(\mathfrak h(x)\) in \eqref{eq:mathfrak_h}, we have
\[
\mathfrak h(x)=c-(\kappa-\eta\bar g)x+O(x^2),
\qquad x\downarrow0.
\]
Thus \(\kappa>\eta\bar g\) implies that \(\mathfrak h\) initially decreases from \(\mathfrak h(0)=c>0\), but it does not by itself guarantee the existence of \(R>0\) such that \(\mathfrak h(R)\le0\). The latter is an additional exponential-integrability requirement used to verify the true-martingale condition.

\section{Fourier-Based Option Pricing}\label{sec:fourier_pricing}
This section develops Fourier-based pricing formulas for European options under the risk-neutral endogenous-activity model. We use the COS method in \cite{fang2009novel} as the main pricing method, since it relies only on the real-axis conditional characteristic function. We also present a damped Carr--Madan inversion in \cite{carr1999option} as a benchmark, subject to additional complex-transform admissibility conditions.

Consider a European option with strike price \(K\) and maturity date \(T>t\). Under the drift restriction in Proposition~\ref{prop:risk_neutral_drift_restriction}, the discounted stock price \(e^{-rt}S_t\) is a local martingale. Under the sufficient exponential integrability condition in Theorem~\ref{thm:true_martingale_sufficient_condition}, together with the integrated-activity bound in Proposition~\ref{prop:integrated_activity_exponential_moment}, the discounted stock price is a true martingale on the option horizon \([t,T]\). Therefore, for a European payoff \(\Pi(S_T)\) satisfying the required integrability condition, its time-\(t\) price is
\begin{equation}\label{eq:European_valuation}
\mathcal V_t=e^{-r\tau}\mathbb E_t\!\left[\Pi(S_T)\right].
\end{equation}

With \(\tau=T-t\), Theorem~\ref{thm:real_axis_riccati} gives the real-axis Fourier--Laplace transform of the joint state \((X_T,\lambda_T)\). Setting \(v=0\) yields the conditional characteristic function of the terminal log-price:
\begin{align}\label{eq:characteristic_function_of_payoff}
\Phi_t(u;T)
&:= \mathbb E_t\!\left[e^{iuX_T}\right]\notag\\
&=\exp\left( iuX_t+\phi(\tau;u,0)+\psi(\tau;u,0)\lambda_t \right),
\qquad u\in\mathbb R,
\end{align}
where \((\phi,\psi)\) solve the real-axis Riccati system in \eqref{eq:generalized_Riccati_system}. Now we define the terminal log-moneyness relative to strike \(K\) by
\begin{equation}\label{eq:terminal_log-moneyness}
Z_T:=X_T-\log K=\log(S_T/K).
\end{equation}
Then the conditional characteristic function of \(Z_T\) is
\begin{equation}\label{eq:characteristic_function_of_log-moneyness}
\varphi_t(u;T,K) := \mathbb E_t\!\left[e^{iuZ_T}\right] = e^{-iu\log K}\Phi_t(u;T), \qquad u\in\mathbb R.
\end{equation}

\subsection{COS Pricing}\label{subsec:cos_pricing}
The COS method approximates the risk-neutral pricing expectation from the real-axis characteristic function of the log-moneyness variable \(Z_T\). Let \([L,U]\) be a truncation interval for \(Z_T\), and let \(N\in\mathbb N\) denote the number of cosine terms retained in the approximation. Define
\begin{equation}\label{eq:COS_frequency}
u_n:=\frac{n\pi}{U-L},
\qquad n=0,\ldots,N-1,
\end{equation}
and the payoff can be written as
\[
\Pi(S_T)=\Pi(Ke^{Z_T}).
\]
We define the payoff cosine coefficients by
\begin{equation}
V_n^{\Pi}:=\frac{2}{U-L}\int_L^U\Pi(Ke^\zeta)\cos\!\left(u_n(\zeta-L)\right)\,d\zeta.
\end{equation}
The COS approximation to the European payoff price is
\begin{align}\label{eq:cos_european_price}
\mathcal V_t^{\mathrm{COS}}
&=e^{-r\tau}\sum_{n=0}^{N-1}{}'\operatorname{Re}\left[e^{-iu_nL}\varphi_t(u_n;T,K)\right]V_n^{\Pi} \notag\\
&=e^{-r\tau}\left\{\frac12 V_0^{\Pi}+\sum_{n=1}^{N-1}\operatorname{Re}\left[e^{-iu_nL}\varphi_t(u_n;T,K)\right]V_n^{\Pi}\right\},
\end{align}
where the prime on the summation means that the \(n=0\) term is weighted by one half.

The truncation interval \([L,U]\) is selected using the cumulant-based COS rule. Let \(\mathfrak c_j^Z=\mathfrak c_j^Z(t,T;K)\) denote the \(j\)-th conditional cumulant of the log-moneyness variable \(Z_T\). Then, we set
\begin{equation}\label{eq:truncation_interval}
[L,U]=\left[\mathfrak c_1^Z-\varsigma\sqrt{\mathfrak c_2^Z+\sqrt{\mathfrak c_4^Z}},\,\mathfrak c_1^Z+\varsigma\sqrt{\mathfrak c_2^Z+\sqrt{\mathfrak c_4^Z}}\right],
\qquad \varsigma>0.
\end{equation}
This is the standard cumulant-based COS truncation rule of \cite{fang2009novel}. In numerical implementation, if \(\mathfrak c_4^Z\) is negative, we use the sign-robust variant described in Section~\ref{subsec:cos_pricing_numerics}. The required cumulants \(\mathfrak c_1^Z,\mathfrak c_2^Z,\mathfrak c_4^Z\) are computed by differentiating the affine Riccati system at the origin; the details are reported in Appendix~\ref{app:cos_cumulants}.

\subsection{Carr--Madan Inversion}\label{subsec:carr_madan_benchmark}
We also report a damped Carr--Madan inversion method for Fourier pricing. Unlike the COS formula, this representation requires the log-price transform on the shifted contour implied by the damped call-price transform in \cite[Eq.~(6)]{carr1999option}:
\[
u=\omega-i(\delta+1),
\qquad \omega\in\mathbb R,
\]
where \(\delta>0\) is the damping parameter. Since this contour lies outside the real-axis domain verified in Theorem~\ref{thm:real_axis_riccati}, we impose the following additional admissibility condition only for the Carr--Madan benchmark.

\begin{assumption}[Carr--Madan admissibility]\label{ass:carr_madan_benchmark_admissibility}
Fix a damping parameter \(\delta>0\) and an option horizon \(\tau=T-t\). Assume that the following conditions hold:
\begin{enumerate}[label=(\roman*)]
\item the positive-tail moment condition is satisfied,
\begin{equation}
\delta+1<M;
\end{equation}
\item for every \(\omega\in\mathbb R\), the complex Riccati system associated with
\[
u=\omega-i(\delta+1), \qquad v=0
\]
admits a non-explosive solution on \([0,\tau]\), so that the complex log-price transform
\[
\Phi_t\bigl(\omega-i(\delta+1);T\bigr)
\]
is finite on the damped contour.
\end{enumerate}
\end{assumption}

\begin{proposition}[Fourier representation of European call prices]\label{prop:fourier_call_pricing}
Suppose that the conditions of Assumption~\ref{ass:carr_madan_benchmark_admissibility} hold and that the damped Carr--Madan integrand is integrable over \(\mathbb R\). Then the time-\(t\) price of a European call option with strike price \(K\) and maturity date \(T\) is given by
\begin{equation}\label{eq:Carr--Madan}
C_t(K,T)=\frac{K^{-\delta}e^{-r\tau}}{\pi}\int_0^\infty\operatorname{Re}\left[K^{-i\omega}\frac{\Phi_t\bigl(\omega-i(\delta+1);T\bigr)}{\delta^2+\delta-\omega^2+i(2\delta+1)\omega}\right]d\omega.
\end{equation}
\end{proposition}

\begin{proof}
See Appendix~\ref{app:Carr_Madan}.
\end{proof}

The corresponding European put price can be recovered from put--call parity. Under the standing risk-neutral valuation assumptions of this section, we have
\begin{equation}\label{eq:put_price}
P_t(K,T)=C_t(K,T)-S_t+e^{-r\tau}K.
\end{equation}
Therefore, it is sufficient to state the damped Fourier representation for call options.

The damped Carr--Madan formula in \eqref{eq:Carr--Madan} is an exact pricing representation whenever the required complex transform domain is admissible. By contrast, the COS formula only uses the real-axis characteristic function verified in Theorem~\ref{thm:real_axis_riccati}, and is therefore treated as the more robust default in the numerical implementation.

\section{Numerical Implementation and Diagnostics}\label{sec:numerical}
In the numerical experiments, the pricing problem is organized around the evaluation of the conditional characteristic function at the initial date \(t=0\). The objective is not to calibrate the model to market option data, but to demonstrate the numerical implementation, verify the stability of the Fourier-based pricing procedure, and illustrate the implied-volatility channels generated by the endogenous-activity mechanism. In addition, when \(\eta=0\), the endogenous feedback channel is switched off, and log-price jumps no longer excite the activity state. The transform equations then reduce to a linear benchmark case, whose explicit solution is reported in Appendix~\ref{app:eta_zero_benchmark}.

The numerical coefficients are set as follows. The spot price and financial coefficients are
\[
S_0=100,\qquad r=0.02,\qquad \sigma=0.12.
\]
The coefficients of the normalized tempered-stable L\'evy shape are
\[
p=0.4,\qquad M=8,\qquad G=4,\qquad \alpha=0.8.
\]
And the activity process coefficients are
\[
\kappa=5,\qquad \bar\lambda=0.08,\qquad \eta=1,\qquad a=1,
\qquad \lambda_0=0.10.
\]
Moreover, strikes are expressed in current log-moneyness,
\begin{equation}\label{eq:current_log-moneyness}
k=\log(K/S_0).
\end{equation}
Although the pricing formulas are defined for arbitrary strike--maturity pairs, the numerical diagnostics are evaluated on the finite strike--maturity grid
\begin{equation}\label{eq:grid}
k\in\{-0.30,-0.25,-0.20,\ldots,0.20,0.25,0.30\},
\qquad
T\in\{1/12,\,1/4,\,1/2,\,1\}.
\end{equation}

\subsection{L\'evy-Integral Term Approximation}\label{subsec:levy_integral_approximation}
We evaluate all L\'evy integrals
\begin{equation}\label{eq:Levy_integral}
\int_{\mathbb R_0}\,(\cdot)\,\bar\nu(dy;\vartheta)
\end{equation}
using the same split quadrature protocol. The negative half-line can be mapped to the positive half-line by the change of variables \(x=-y\). For numerical evaluation, the half-line is therefore truncated at \(y_{\mathrm{cut}}\) and decomposed as
\[
(0,y_{\mathrm{cut}})=(0,1)\cup(1,y_{\mathrm{cut}}).
\]
The small-jump integrals are evaluated by Gauss--Jacobi quadrature, whereas the tail integrals are evaluated by fixed-node Gauss--Legendre quadrature. Both are standard Gaussian quadrature rules; see \cite{davis2007methods} for numerical integration methods and \cite{golub1969calculation} for the computation of Gaussian quadrature rules. In computation, we use
\[
n_{\mathrm{small}}=128,\qquad
n_{\mathrm{tail}}=128,\qquad
y_{\mathrm{cut}}=10.
\]

Under this quadrature protocol, \(\bar g\) in \eqref{eq:bar_g} is
\[
\int_{\mathbb R_0}g(y)\bar\nu(dy;\vartheta)=0.952397.
\]
Hence the feedback ratio is
\(
\eta\bar g/\kappa=0.190479<1,
\)
which satisfies the mean-subcriticality condition in \eqref{eq:subcriticality_condition}. By Proposition~\ref{prop:mean_subcriticality}, the stationary mean activity
level is
\[
\frac{\kappa\bar\lambda}{\kappa-\eta\bar g}=0.098824,
\]
close to the initial activity level \(\lambda_0=0.10\). This avoids starting the numerical exercise from an activity state that is far from its long-run mean.

Similarly, the value of \(\chi_J\) in \eqref{eq:chi_J} is
\[
\int_{\mathbb R_0}\left(e^y-1-y\mathbf 1_{\{|y|<1\}}\right)\bar\nu(dy;\vartheta)=0.470998.
\]
This value is used in the risk-neutral drift restriction and in the Riccati right-hand side. Moreover, \(\mathcal J_u(z)\) defined in Appendix~\ref{app:proof_real_axis_riccati} is another L\'evy-integral term in Riccati. Similarly, given the Fourier frequency \(u\in\mathbb R\), we obtain
\[
\mathcal J_u(z)=\mathcal J_u^{+}(z)+\mathcal J_u^{-}(z),
\]
where
\[
\mathcal J_u^+(z)=C_+\int_0^\infty\left[e^{iuy+\eta z g(y)}-1-iuy\mathbf 1_{\{y<1\}}\right]e^{-My}y^{-1-\alpha}\,dy,
\]
and
\[
\mathcal J_u^-(z)=C_-\int_0^\infty\left[e^{-iuy+\eta z g(y)}-1+iuy\mathbf 1_{\{y<1\}}\right]e^{-Gy}y^{-1-\alpha}\,dy.
\]
Here, \(C_+=pM^{2-\alpha}/\Gamma(2-\alpha)=5.282573\) and \(C_-=(1-p)G^{2-\alpha}/\Gamma(2-\alpha)=3.449060\).

We also verify the sufficient true-martingale condition used for risk-neutral valuation. The \(H_J\) in Theorem~\ref{thm:true_martingale_sufficient_condition} is
\[
\int_{\mathbb R_0}\left(ye^y-e^y+1\right)\bar\nu(dy;\vartheta)=0.472517.
\]
Taking \(R=0.15\), we have
\[
\mathfrak h(0.15)=0.472517-5\times0.15+\int_{\mathbb R_0}\left(e^{\eta(0.15)g(y)}-1\right)\bar\nu(dy;\vartheta)=-0.133812<0.
\]
Thus the sufficient condition in Proposition~\ref{prop:integrated_activity_exponential_moment} is satisfied for numerical setup, and the discounted stock price is a true martingale on the finite option horizons considered below.

\begin{remark}
Gauss--Jacobi quadrature is used on the small-jump region because, by \eqref{eq:origin_behavior_g}, the non-kernel factors in the L\'evy integrals in \eqref{eq:Levy_integral} are \(O(y^2)\) near the origin after compensation. Hence, after multiplication by the small-jump kernel \(|y|^{-1-\alpha}\), the corresponding integrands have the endpoint behavior
\[
O(y^2)|y|^{-1-\alpha}=O(|y|^{1-\alpha}).
\]
This is the endpoint weight incorporated by the Gauss--Jacobi rule. By contrast, the tail integrals on \((1,y_{\mathrm{cut}})\) are smooth on a finite interval and are evaluated by fixed-node Gauss--Legendre quadrature.
\end{remark}

As a truncation-robustness check, we vary
\[
y_{\mathrm{cut}}\in\{20,30,40,50\}.
\]
We compute the maximum absolute discrepancy relative to \(y_{\mathrm{cut}}=10\) across
\[
\bar g,\quad \chi_J,\quad H_J,\quad\int_{\mathbb R_0}\left(e^{\eta(0.15)g(y)}-1\right)\bar\nu(dy;\vartheta),\quad\mathfrak h(0.15).
\]
The largest discrepancy is \(3.886\times 10^{-15}\).

\subsection{Riccati Solver and Characteristic-Function Evaluation}\label{subsec:riccati_solver_numerics}
After fixing the L\'evy-integral quadrature, we solve the Riccati system in \eqref{eq:generalized_Riccati_system} for each maturity and each real Fourier frequency required by the COS implementation. The real Fourier axis is the relevant domain for the COS method. The complex-valued Riccati system is implemented as a real four-dimensional ODE for
\[
\left(
\operatorname{Re}\phi,
\operatorname{Im}\phi,
\operatorname{Re}\psi,
\operatorname{Im}\psi
\right).
\]
We integrate the Riccati system using the DOP853 adaptive Runge--Kutta method from \texttt{SciPy} \cite{virtanen2020scipy}. The solver tolerances are
\[
\texttt{rtol}=10^{-9},
\qquad
\texttt{atol}=10^{-11}.
\]
At each ODE step, the right-hand side \(\mathcal K_1(u,\psi)\) is evaluated using the quadrature protocol described above. Once \((\phi,\psi)\) has been obtained, the conditional characteristic function is evaluated as
\[
\Phi_0(u;T)=\exp\left(iu\log S_0+\phi(T;u,0)+\psi(T;u,0)\lambda_0\right).
\]
This quantity is the direct input into the Fourier pricing formulas.

As a solver-robustness check, we recompute the Riccati system in \eqref{eq:generalized_Riccati_system} using the implicit Radau method, commonly used for stiff ODEs \cite{wanner1996solving}, and compare it with the DOP853 solution to ensure that the computed Riccati solutions are not solver-specific. The comparison is carried out on the diagnostic grid
\begin{equation}\label{eq:u_grid}
u\in\{0,5,25,75,150\}.
\end{equation}
For each pair \((u,T)\), we obtain
\[
\max_{u,T} \left| \Phi^{\mathrm{DOP853}}_0(u;T)-\Phi^{\mathrm{Radau}}_0(u;T) \right| = 1.531\times10^{-13}.
\]
Thus the COS characteristic-function evaluations are solver-robust at the reported tolerances.

\subsection{Numerical Results of COS}\label{subsec:cos_pricing_numerics}
The COS expansion order is fixed at \(N=256\) in computation. For each maturity \(T\) and strike price \(K\), the baseline COS truncation interval follows the cumulant rule in \eqref{eq:truncation_interval}. In the numerical implementation, we use the sign-robust variant
\[
L_{0,T;K}=\mathfrak c_1^Z(0,T;K)-\varsigma\sqrt{\mathfrak c_2^Z(0,T;K)+\sqrt{|\mathfrak c_4^Z(0,T;K)|}},
\]
\[
U_{0,T;K}=\mathfrak c_1^Z(0,T;K)+\varsigma\sqrt{\mathfrak c_2^Z(0,T;K)+\sqrt{|\mathfrak c_4^Z(0,T;K)|}},
\]
with \(\varsigma=10\). By Appendix~\ref{app:cos_cumulants} and current log-moneyness \(k\) in \eqref{eq:current_log-moneyness}, the required cumulants are
\[
\mathfrak c_1^Z(0,T;K)=-k+a_1(T)+\lambda_0b_1(T),
\]
\[
\mathfrak c_2^Z(0,T;K)=a_2(T)+\lambda_0b_2(T),
\]
\[
\mathfrak c_4^Z(0,T;K)=a_4(T)+\lambda_0b_4(T).
\]
Thus the interval width is maturity-specific, while the interval center shifts with \(k\). The resulting second cumulants are positive and increase with maturity.

At the at-the-money strike \(k=0\), the corresponding truncation intervals and COS prices are reported in Table~\ref{tab:cos_basic_diagnostics}. Across the strike grid \(k\) in \eqref{eq:grid}, the computed call prices are decreasing in \(K\), as required by no-arbitrage monotonicity.

\begin{table}[htbp]
\centering
\caption{Basic COS truncation and pricing diagnostics at \(k=0\)}
\label{tab:cos_basic_diagnostics}
\renewcommand{\arraystretch}{1.18}
\setlength{\extrarowheight}{2pt}
\begin{tabular}{ccccc}
\toprule
\(T\) & \(L_{0,T;100}\) & \(U_{0,T;100}\) & \(c_2^Z(0,T;100)\) & \(C_0^{\mathrm{COS}}\) \\
\hline
\(1/12\) & \(-2.095910\) & \(2.089985\) & \(9.547241\times 10^{-3}\) & \(3.316852\) \\
\(1/4\)  & \(-3.107193\) & \(3.089483\) & \(2.870443\times 10^{-2}\) & \(6.244772\) \\
\(1/2\)  & \(-4.051654\) & \(4.016348\) & \(5.752622\times 10^{-2}\) & \(9.189160\) \\
\(1\)    & \(-5.298573\) & \(5.228174\) & \(1.152733\times 10^{-1}\) & \(13.461744\) \\
\bottomrule
\end{tabular}
\end{table}

To assess sensitivity of the truncation multiplier, we vary it over
\[
\varsigma\in\{8,9,11,12\}.
\]
The maximum absolute price deviation is
\[
\max_{k,T,\varsigma}\left|C_0^{\mathrm{COS}}(k,T;\varsigma)-C_0^{\mathrm{COS}}(k,T;10)\right|=2.352225\times 10^{-6}.
\]

\subsection{Implied-Volatility Inversion}\label{subsec:iv_inversion_numerics}
The COS prices are converted into Black--Scholes implied-volatility (BS IV) units strike by strike and maturity by maturity. For each pair \((k,T)\), the COS method yields a call price \(C_0^{\mathrm{COS}}(K,T)\). The corresponding put price is obtained from \eqref{eq:put_price},
\[
P_0^{\mathrm{COS}}(K,T)=C_0^{\mathrm{COS}}(K,T)-S_0+e^{-rT}K .
\]
For numerical stability, implied volatilities are computed from out-of-the-money (OTM) option prices:
\[
\mathcal V_0^{\mathrm{OTM}}(K,T)=
\begin{cases}
P_0^{\mathrm{COS}}(K,T), & k<0,\\[0.3em]
C_0^{\mathrm{COS}}(K,T), & k\ge 0.
\end{cases}
\]
The Black--Scholes implied volatility is then defined as the unique solution
\[
\mathcal V_0^{\mathrm{BS},\mathrm{OTM}}\left(S_0,K,r,T,\sigma_{\mathrm{IV}}(k,T)\right)=\mathcal V_0^{\mathrm{OTM}}(K,T).
\]

Applying this inversion to every point on the grid in \eqref{eq:grid} gives the model-implied volatility surface
\[
(k,T)\longmapsto \sigma_{\mathrm{IV}}(k,T).
\]
Figure~\ref{fig:model_implied_surface_3d} shows the pronounced short-maturity curvature and the gradual flattening of the implied-volatility profile as maturity increases. In addition, Figure~\ref{fig:model_implied_smiles} reports the model-implied volatility smiles, which display a pronounced short-maturity curvature. For \(T=1/12\), the implied volatilities at \(k=-0.30\), \(k=0\), and \(k=0.30\) are
\[
0.516470,\qquad 0.281031,\qquad 0.456148,
\]
respectively. The left wing is higher than the right wing, consistent with the negatively skewed jump specification generated by \(G=4<M=8\). The left--right spread
\begin{equation}\label{eq:left-right_spread}
\sigma_{\mathrm{IV}}(-0.30,T)-\sigma_{\mathrm{IV}}(0.30,T)
\end{equation}
decreases from \(0.060322\) at \(T=1/12\) to \(0.041382\) at \(T=1/4\), \(0.029886\) at \(T=1/2\), and \(0.019965\) at \(T=1\). This indicates that the short-maturity skew is the steepest and that the smile becomes flatter as maturity increases.

\begin{figure}[htbp]
\centering
\includegraphics[width=1.0\textwidth]{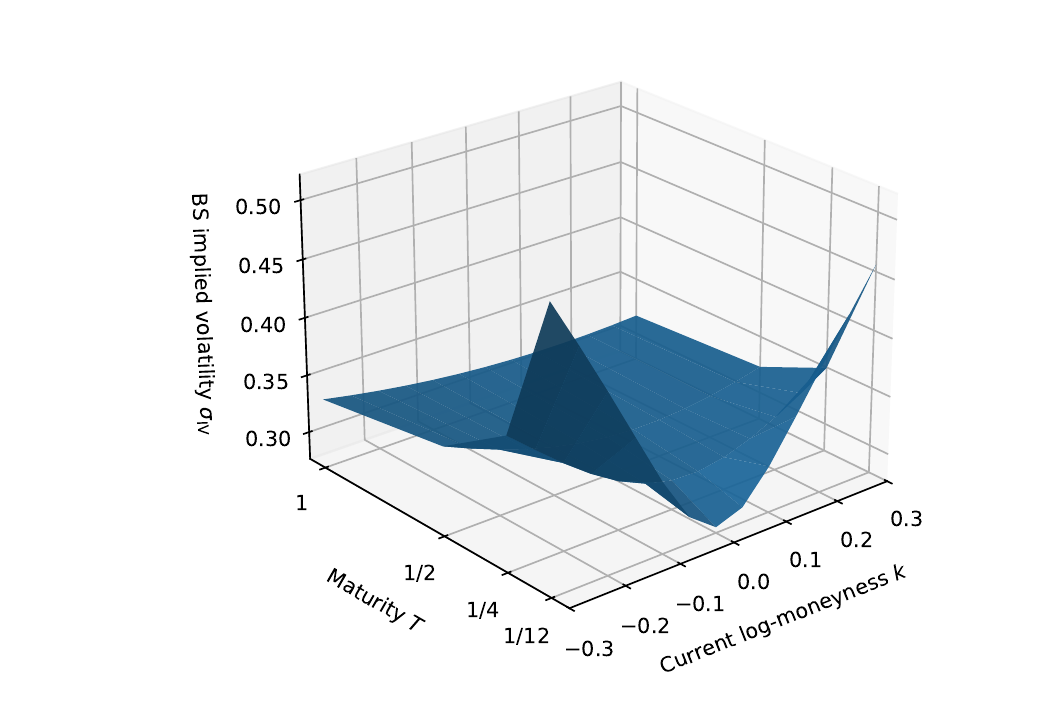}
\caption{Model-implied volatility surface.}
\label{fig:model_implied_surface_3d}
\end{figure}

\begin{figure}[htbp]
\centering
\includegraphics[width=0.80\textwidth]{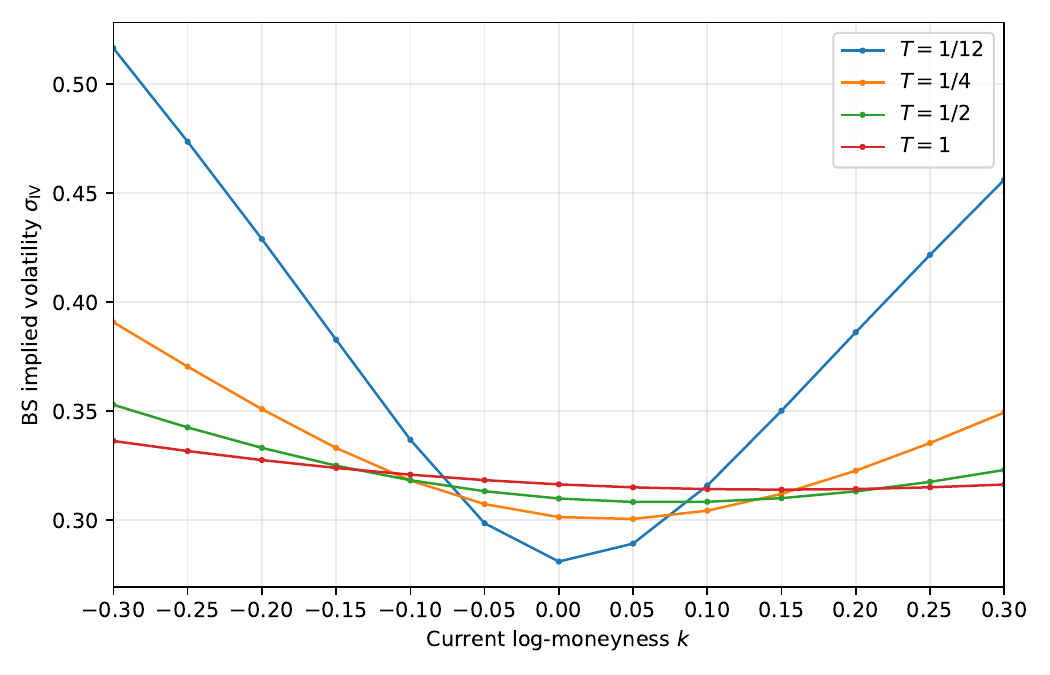}
\caption{Model-implied volatility smiles.}
\label{fig:model_implied_smiles}
\end{figure}

\subsection{Mechanism Decomposition}\label{subsec:mechanism_decomposition}
The preceding subsection shows that, in this numerical setup, the endogenous-activity model produces a pronounced short-maturity smile, a higher left wing, and a narrowing left--right spread. We now decompose these effects into two dynamic channels of the activity process, the initial activity level \(\lambda_0\) and the endogenous feedback parameter \(\eta\).

Firstly, Figure~\ref{fig:lambda0_channel_smile} illustrates the effect of \(\lambda_0\) on the implied-volatility smile on the grid in \eqref{eq:grid}. The maturity is fixed at \(T=1/12\), and all parameters except \(\lambda_0\) are fixed. Increasing \(\lambda_0\) from \(0.05\) to \(0.20\) shifts the entire smile upward. In particular, the at-the-money implied volatility increases from \(0.220060\) to \(0.383056\), while the mean implied volatility across the reported moneyness grid increases from \(0.330505\) to \(0.460399\). This supports that the current activity state mainly controls the near-term volatility level. The cross-sectional shape of the smile is largely preserved, because the jump-size distribution is unchanged across the three cases. The  left--right spread in \eqref{eq:left-right_spread} also increases moderately, from \(0.053665\) at \(\lambda_0=0.05\) to \(0.067318\) at \(\lambda_0=0.20\), reflecting the higher near-term amplification of jump risk under a larger current activity level.

Secondly, Figure~\ref{fig:eta_channel_skew_profile} examines how the \(\eta\) affects the maturity profile of the implied-volatility skew. To isolate this effect from changes in the unconditional activity level, we use a stationary-mean-matched comparison. For each value of \(\eta\), \(\bar\lambda\) is adjusted so that
\[
\frac{\kappa\bar\lambda(\eta)}{\kappa-\eta\bar g}
\]
remains fixed at its benchmark value \(0.098824\). This gives \(\bar\lambda(0)=0.098824\) and \(\bar\lambda(0.5)=0.089412\). The left--right spread declines with maturity for all values of \(\eta\), but the decline is slower when the feedback parameter is larger. This indicates that endogenous feedback mainly affects the persistence of jump-induced skew across maturities rather than merely shifting the volatility level.

\begin{remark}
At \(T=1/12\), the spread is slightly higher for \(\eta=0\) than for \(\eta=1\) (\(0.061578\) versus \(0.060322\)). This small short-maturity difference is not the main feedback effect. At very short horizons, the smile is dominated by the current activity state and the instantaneous jump-size distribution. The effect of \(\eta\) becomes clearer at medium and long maturities, where feedback has time to propagate jump-induced activity.
\end{remark}

\begin{figure}[htbp]
\centering
\includegraphics[width=0.80\textwidth]{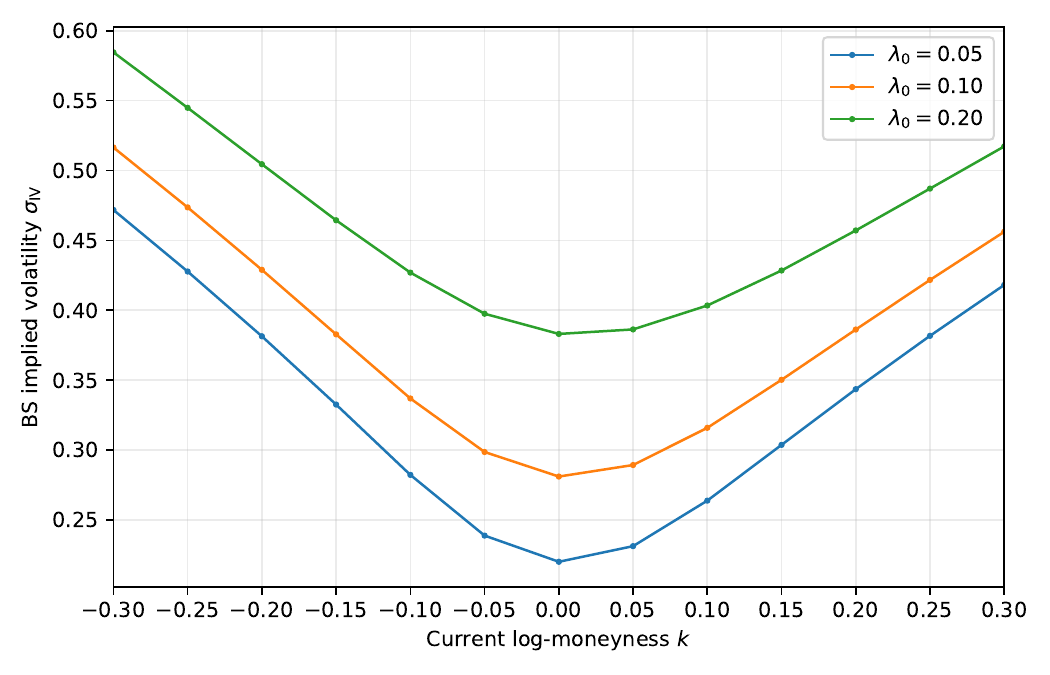}
\caption{Initial-activity channel.}
\label{fig:lambda0_channel_smile}
\end{figure}

\begin{figure}[htbp]
\centering
\includegraphics[width=0.80\textwidth]{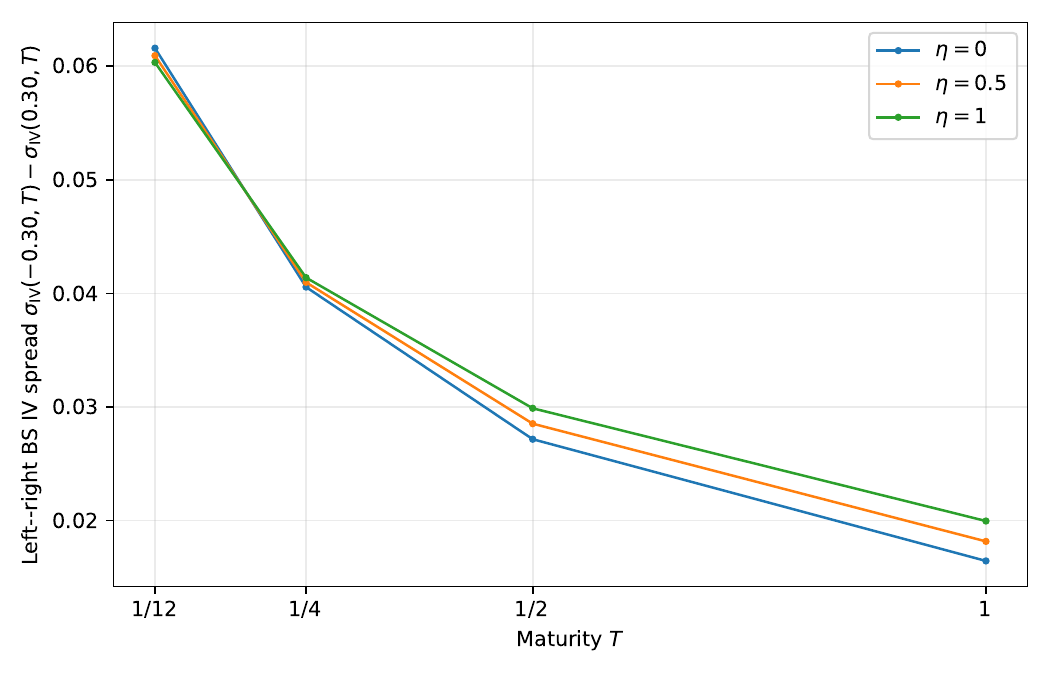}
\caption{Endogenous-feedback channel.}
\label{fig:eta_channel_skew_profile}
\end{figure}

\subsection{Carr--Madan Check}\label{subsec:carr_madan_benchmark_numerics}
We compare COS prices with prices computed from the Carr--Madan inversion formula. For the Carr--Madan calculation, we use the damping parameter \(\delta=1.5\). Under the present jump-tail specification, \(M=8\), we have
\[
\delta+1=2.5<M=8.
\]

The Carr--Madan inversion in \eqref{eq:Carr--Madan} involves an outer integral over the frequency variable \(\omega\), which is truncated to \([0,\omega_{\max}=150]\) and evaluated by Gauss--Legendre quadrature with \(512\) quadrature nodes. At each quadrature node, the complex-valued characteristic function \(\Phi_0(\omega-i(\delta+1);T)\) is obtained by solving the Riccati system. The Lévy integrals appearing inside the Riccati right-hand side are evaluated using the same Gauss--Jacobi/Gauss--Legendre quadrature protocol described above.

The maximum price discrepancy is
\[
\max_{k,T}\left|C_0^{\mathrm{CM}}(K,T)-C_0^{\mathrm{COS}}(K,T)\right|=1.304609\times10^{-7}.
\]
After Black--Scholes inversion, the maximum implied-volatility discrepancy is
\[
\max_{k,T}\left|\sigma_{\mathrm{IV}}^{\mathrm{CM}}(k,T)-\sigma_{\mathrm{IV}}^{\mathrm{COS}}(k,T)\right|=1.273328\times10^{-7}.
\]
The largest price and implied-volatility discrepancy both occur at \(T=1/12\) and \(k=0.30\).

\subsection{Monte Carlo validation}\label{subsec:mc_validation}
As an independent validation of the Riccati-based COS prices, we simulate the coupled process \((X_t,\lambda_t)\) directly. Since \(\bar\nu\) has infinite activity, we use a cutoff \(\varepsilon\): jumps with \(|y|<\varepsilon\) are replaced by a Gaussian approximation with variance rate
\[
s_\varepsilon^2=\int_{|y|<\varepsilon}y^2\,\bar\nu(dy),
\]
while their average contribution to the activity dynamics is absorbed through
\[
\bar g_{\varepsilon}=\int_{|y|<\varepsilon}g(y)\,\bar\nu(dy).
\]
The finite-activity jumps are simulated by Euler thinning. On \([t_n,t_{n+1})\), with \(t_n=n\Delta\) and simulated activity \(\lambda_n\), the number of explicit jumps is drawn from
\[
\operatorname{Poisson}\bigl(\lambda_n\bar\nu(|y|\ge\varepsilon)\Delta\bigr),
\]
and their sizes are sampled from
\[
\frac{\mathbf 1_{\{|y|\ge\varepsilon\}}\bar\nu(dy)}{\bar\nu(|y|\ge\varepsilon)}.
\]
The log-price drift is adjusted for this Gaussian replacement so that the truncated simulation preserves the risk-neutral stock-price martingale identity.

On the maturity grid in \eqref{eq:grid}, the main simulation uses
\[
N_{\mathrm{MC}}=300{,}000,
\qquad\varepsilon=0.02,
\qquad\Delta=1/504.
\]
The simulated terminal states satisfy the martingale identity
\[
\mathbb E[e^{-rT}S_T]=S_0
\]
within Monte Carlo error, with maximum absolute standardized discrepancy \(0.856\). The sample mean of \(\lambda_T\) is also consistent with the finite-horizon mean formula in \eqref{eq:first_moment_m}, with maximum standardized discrepancy \(0.524\). Moreover, the empirical characteristic function
\[
\widehat\Phi_{\mathrm{MC}}(u;T)=\frac1{N_{\mathrm{MC}}}\sum_{j=1}^{N_{\mathrm{MC}}}e^{iuX_T^{(j)}}
\]
agrees with the Riccati characteristic function on the nonzero frequency grid in \eqref{eq:u_grid}, with a maximum standardized discrepancy of \(1.28\).

For option prices, we use the discounted stock price \(e^{-rT}S_T\) as a control variate. Table~\ref{tab:mc_validation} compares the COS prices with the resulting Monte Carlo control-variate (MC-CV) estimates. In all cases, the COS price lies within the corresponding Monte Carlo confidence interval, whose half-width is reported as \(1.96\,\mathrm{standard\,errors\,(SE)}\).

\begin{table}[htbp]
\centering
\caption{Monte Carlo validation of COS prices.}
\label{tab:mc_validation}
\begin{tabular}{rrrrr}
\toprule
\(k\) & \(T\) & \(C^{\mathrm{COS}}\) &
\(C^{\mathrm{MC-CV}}\) & \(1.96\,\mathrm{SE}\)\\
\midrule
\(-0.30\) & \(1/12\) & \(26.143560\) & \(26.143724\) & \(0.004433\) \\
\( 0.00\) & \(1/12\) & \( 3.316852\) & \( 3.318835\) & \(0.012390\) \\
\( 0.30\) & \(1/12\) & \( 0.061762\) & \( 0.060264\) & \(0.004895\) \\
\( 0.00\) & \(1/4\)  & \( 6.244772\) & \( 6.232020\) & \(0.019786\) \\
\( 0.00\) & \(1/2\)  & \( 9.189160\) & \( 9.204432\) & \(0.027097\) \\
\( 0.30\) & \(1\)    & \( 3.727424\) & \( 3.726130\) & \(0.036096\) \\
\bottomrule
\end{tabular}
\end{table}

\section{Conclusions and Future Challenges} \label{sec:future_challenge}
The model developed in this work is positioned between infinite-activity return models and event-based self-exciting jump models. As summarized in Table~\ref{tab:model_positioning}, its main distinction is to combine infinite jump activity with an activity scale that is endogenous to realized price jumps. This positioning preserves affine tractability, but also points to two natural extensions: more flexible excitation mechanisms and applications beyond vanilla option pricing.

\begin{table}[t]
\centering
\footnotesize
\setlength{\tabcolsep}{2pt}
\renewcommand{\arraystretch}{1.16}
\caption{Model positioning}
\label{tab:model_positioning}
\begin{tabularx}{\textwidth}{@{}L{0.33\textwidth}Y C{0.15\textwidth} C{0.13\textwidth}@{}}
\toprule
Model class
& Activity driver
& Endogeneity
& Jump activity\\
\midrule
Stochastic-clock CGMY/time-changed L\'evy models~\cite{carr2002fine}
& Independent or separately specified activity clock
& Exogenous
& Infinite\\

Affine jump models with stochastic intensity~\cite{duffie2000transform}
& Exogenous factor stochastic differential equation (SDE)
& Exogenous
& Finite\\

Event-based Hawkes models~\cite{errais2010affine,ait2015modeling}
& Event counts
& Endogenous
& Finite\\

\addlinespace[0.2em]
\textbf{This work}
& Own realized price jumps through a bounded quadratic-variation-type response
& \textbf{Endogenous}
& \textbf{Infinite}\\
\bottomrule
\end{tabularx}
\end{table}

The first extension is asymmetric excitation. The excitation kernel
\[
g(y)=1-e^{-ay^2}
\]
depends on the jump size only through \(y^2\). The model therefore contains no direct activity-leverage channel: conditional on jump magnitude, positive and negative jumps generate the same increase in future activity. Any stronger future activity or volatility response associated with downside moves can only enter indirectly through the asymmetry of the jump-size law, for example through \(G<M\), which gives the negative tail slower exponential tempering. Allowing \(g\) to assign larger excitation to negative jumps would separate dynamic leverage from the static asymmetry of the jump-size law. Such an extension would preserve the affine and Riccati structure as long as the asymmetric kernel remains bounded and satisfies
\[
\int_{\mathbb R_0} g(y)\,\bar\nu(dy;\vartheta)<\infty .
\]
The main difficulty would therefore be empirical rather than analytical: the maturity structure of option-implied skew, realized semivariances, variance-swap data, and volatility-index data may all be needed to disentangle jump-size asymmetry from asymmetric activity feedback.

The second extension is the pricing of variance and volatility derivatives. The normalization used in this work makes \(\lambda_t\) the predictable rate of jump-induced quadratic variation. Thus realized-variance payoffs are directly linked to the integrated activity \[ \int_t^T \lambda_s\,ds, \] in addition to the continuous Brownian contribution \(\sigma^2(T-t)\). This suggests a direct route to variance-swap rates, volatility-index-type quantities, and options on realized variance or volatility. A full treatment would require transforms or stable numerical schemes for integrated activity under endogenous feedback. We leave these extensions to future work.

\begin{appendices}
\renewcommand{\thesection}{\arabic{section}}
\section{}\label{app:Levy_shape_results}
\textbf{Proof of Theorem~\ref{thm:normalized-shape-admissibility}:} Near zero, the exponential tempering terms satisfy
\[
e^{-My}\to 1,
\qquad
e^{-G|y|}\to 1.
\]
And the density of \(\bar\nu\) behaves like \(|y|^{-1-\alpha}dy\). Therefore,
\[
\int_{|y|<1}y^2\bar\nu(dy;\vartheta)\asymp\int_0^1 y^{1-\alpha}dy<\infty
\]
because \(\alpha<2\). Away from zero, the exponential factors \(e^{-My}\) and \(e^{-G|y|}\) ensure integrability. Thus
\[
\int_{\mathbb R_0}(1\wedge y^2)\bar\nu(dy;\vartheta)<\infty,
\]
which means \(\bar\nu(dy;\vartheta)\) is a valid L\'evy measure. Moreover,
\[
\bar\nu(|y|<1)\asymp\int_0^1 y^{-1-\alpha}dy=\infty,
\qquad\alpha>0.
\]
Hence the jump component has infinite activity.

The finite-variation condition for the small jumps is
\[
\int_{|y|<1}|y|\bar\nu(dy;\vartheta)<\infty.
\]
Using the same near-zero behavior, we obtain
\[
\int_{|y|<1}|y|\bar\nu(dy;\vartheta)\asymp\int_0^1 y^{-\alpha}dy.
\]
This integral is finite if and only if \(\alpha<1\). Hence the small jumps have finite path variation exactly when \(\alpha<1\).

For \(n\ge2\), the positive-side contribution is
\[
\int_{y>0}y^n\bar\nu(dy;\vartheta)=p\frac{M^{2-\alpha}}{\Gamma(2-\alpha)}\int_0^\infty y^{n-1-\alpha}e^{-My}dy.
\]
By
\[
\int_0^\infty y^{n-\alpha-1}e^{-My}dy=\Gamma(n-\alpha)M^{\alpha-n},
\]
we obtain
\[
\int_{y>0}y^n\bar\nu(dy;\vartheta)=p\frac{\Gamma(n-\alpha)}{\Gamma(2-\alpha)}M^{2-n}.
\]
Similarly, with \(z=|y|=-y\), the negative-side contribution is
\[
\int_{y<0}y^n\bar\nu(dy;\vartheta)=(-1)^n(1-p)\frac{\Gamma(n-\alpha)}{\Gamma(2-\alpha)}G^{2-n}.
\]
Therefore,
\[
c_n(\vartheta):=\int_{\mathbb R_0}y^n\bar\nu(dy;\vartheta)=\frac{\Gamma(n-\alpha)}{\Gamma(2-\alpha)}\left[pM^{2-n}+(-1)^n(1-p)G^{2-n}\right]<\infty.
\]
Taking \(n=2\) gives
\[
c_2(\vartheta)=\int_{\mathbb R_0}y^2\bar\nu(dy;\vartheta)=p+(1-p)=1.
\]
In particular,
\[
\int_{y>0}y^2\bar\nu(dy;\vartheta)=p,
\qquad
\int_{y<0}y^2\bar\nu(dy;\vartheta)=1-p.
\]

\section{}\label{app:well_posed_activity}
\textbf{Proof of Proposition~\ref{prop:well_posed_activity}:} The activity equation can be written as the one-dimensional Poisson-random-measure equation
\[
\lambda_t = \lambda_0 + \int_0^t \kappa(\bar\lambda-\lambda_s)\,ds + \int_0^t\int_{\mathbb R_0}\int_0^\infty \eta g(y)\mathbf 1_{\{\gamma\le \lambda_{s-}\}} \,\mathcal N(ds,dy,d\gamma).
\]
This is a one-dimensional jump-type stochastic equation of the general form studied in \cite[Section~5]{li2012strong}, with no Brownian
coefficient and no compensated jump coefficient. In their notation, we take
\[
U_1=\mathbb R_0\times[0,\infty),
\qquad\mu_1(dy,d\gamma)=\bar\nu(dy;\vartheta)d\gamma,
\]
and set
\[
\sigma\equiv0,
\qquad
g_0\equiv0,
\qquad
g_1(x;y,\gamma)=\eta g(y)\mathbf 1_{\{\gamma\le x\}},
\quad x\ge0.
\]
The drift item is \(x\mapsto\kappa(\bar\lambda-x)\), which admits the
decomposition
\[
\kappa(\bar\lambda-x)=\kappa\bar\lambda-\kappa x.
\]
At the same time, \(\kappa x\) is continuous and non-decreasing. Moreover, for \(x,x'\ge0\),
\[
\begin{aligned}
\int_{U_1}
|g_1(x;u)-g_1(x';u)|\,\mu_1(du)
&=
\eta
\int_{\mathbb R_0}g(y)\bar\nu(dy;\vartheta)
\int_0^\infty
\left|
\mathbf 1_{\{\gamma\le x\}}
-
\mathbf 1_{\{\gamma\le x'\}}
\right|d\gamma  \\
&=
\eta\bar g |x-x'|.
\end{aligned}
\]
Thus the non-compensated jump coefficient satisfies the required integrated Lipschitz condition. It also satisfies the corresponding linear-growth bound, since
\[
\int_{U_1}|g_1(x;u)|\,\mu_1(du)
=
\eta\bar g x,
\qquad x\ge0.
\]
Therefore the standard strong-existence and pathwise-uniqueness theorem for
jump-type stochastic equations applies.

It remains to prove the stated finite-horizon integrability. For \(n\ge1\), define
\[
\rho_n
:=
\inf\left\{
t\ge0:
\int_0^t\lambda_s\,ds\ge n
\right\},
\]
with the convention \(\inf\varnothing=\infty\). On the stopped interval
\([0,t\wedge\rho_n]\), the compensation formula is legitimate because
\[
\int_0^{t\wedge\rho_n}\int_{\mathbb R_0}\int_0^\infty
g(y)\mathbf 1_{\{\gamma\le\lambda_{s-}\}}
\,d\gamma\,\bar\nu(dy;\vartheta)\,ds
=
\bar g\int_0^{t\wedge\rho_n}\lambda_s\,ds
\le \bar g n .
\]
Taking expectations in the stopped equation gives
\[
\mathbb E[\lambda_{t\wedge\rho_n}]
=
\lambda_0
+
\kappa\bar\lambda\,\mathbb E[t\wedge\rho_n]
-
\kappa\mathbb E\!\left[\int_0^{t\wedge\rho_n}\lambda_s\,ds\right]
+
\eta\bar g\,
\mathbb E\!\left[\int_0^{t\wedge\rho_n}\lambda_s\,ds\right].
\]
Therefore, for \(0\le t\le T\),
\[
\mathbb E[\lambda_{t\wedge\rho_n}]
\le
\lambda_0+\kappa\bar\lambda T
+
\eta\bar g
\int_0^t
\mathbb E[\lambda_{s\wedge\rho_n}]\,ds .
\]
By Gronwall's inequality,
\[
\sup_{0\le t\le T}\mathbb E[\lambda_{t\wedge\rho_n}]
\le
(\lambda_0+\kappa\bar\lambda T)\exp(\eta\bar g T)
=:C_T,
\]
where \(C_T\) is independent of \(n\). Hence
\[
\mathbb E\!\left[
\int_0^{T\wedge\rho_n}\lambda_s\,ds
\right]
=
\int_0^T
\mathbb E\!\left[
\mathbf 1_{\{s\le\rho_n\}}\lambda_s
\right]ds
\le
\int_0^T
\mathbb E[\lambda_{s\wedge\rho_n}]\,ds
\le
TC_T .
\]
Letting \(n\to\infty\) and using monotone convergence yields
\[
\mathbb E\!\left[\int_0^T\lambda_s\,ds\right]<\infty .
\]
Since the random variable \(\int_0^T\lambda_s\,ds\) is nonnegative, finite expectation also implies
\[
\int_0^T\lambda_s\,ds<\infty
\qquad\text{a.s.}
\]
This proves the claim.

\section{}\label{app:proof_stationarity_geometric_ergodicity}
\textbf{Proof of Proposition~\ref{prop:stationary_activity}:} Since the jump measure \(\mu\) has predictable compensator
\(\lambda_{t-}\bar\nu(dy;\vartheta)\,dt\), the activity process is a
time-homogeneous Markov process on \([0,\infty)\). For
\(V\in C^1([0,\infty))\) such that the following integral is finite, its
extended generator is
\[
(\mathcal A V)(\lambda)=\kappa(\bar\lambda-\lambda)V'(\lambda)+\lambda\int_{\mathbb R_0}\bigl[V(\lambda+\eta g(y))-V(\lambda)\bigr]\bar\nu(dy;\vartheta).
\]
For Lipschitz \(V\), the integral is finite because
\[
\bigl|V(\lambda+\eta g(y))-V(\lambda)\bigr|\le\operatorname{Lip}(V)\eta g(y)\qquad\text{and}\qquad\bar g<\infty .
\]

We first verify a geometric Foster--Lyapunov drift condition. Let
\[
\mathcal W(\lambda)=1+\lambda .
\]
Then
\[
(\mathcal A\mathcal W)(\lambda)
=
\kappa(\bar\lambda-\lambda)+\lambda\eta\bar g
=
-(\kappa-\eta\bar g)\mathcal W(\lambda)
+\kappa\bar\lambda+(\kappa-\eta\bar g).
\]
Hence, for any \(\mathfrak b\in(0,\kappa-\eta\bar g)\), there exist
\(\mathfrak a<\infty\) and a compact set \(D\subset[0,\infty)\) such that
\[
\mathcal A\mathcal W(\lambda)
\le
-\mathfrak b\mathcal W(\lambda)
+
\mathfrak a\,\mathbf 1_D(\lambda),
\qquad
\lambda\ge0 .
\]
Thus \(\mathcal W\) is a norm-like Lyapunov function satisfying a geometric drift condition.

Because \(\eta g(y)\ge0\), variation of constants gives
\[
\lambda_t=\bar\lambda+(\lambda_0-\bar\lambda)e^{-\kappa t}+\eta\int_0^t e^{-\kappa(t-s)}\int_{\mathbb R_0}g(y)\,\mu(ds,dy)\ge\bar\lambda+(\lambda_0-\bar\lambda)e^{-\kappa t}.
\]
Hence \(\bar\lambda\) is an attracting lower barrier; in particular, once \(\lambda_0\ge\bar\lambda\), the path remains in \([\bar\lambda,\infty)\), and the recurrent state space supporting \(\pi_\lambda\) is \([\bar\lambda,\infty)\). On \((\bar\lambda,\infty)\), the jump kernel maps \(y\) into the positive
increment \(\eta g(y)\). Since
\[
x\mapsto\eta g(x)
\]
is \(C^1\) and strictly increasing on \((0,\infty)\), and since \(\bar\nu\) has positive densities on the half-lines, the induced activity-jump measure has an absolutely continuous component on \((0,\eta)\). Together with the deterministic flow
\[
\lambda\mapsto\bar\lambda+(\lambda-\bar\lambda)e^{-\kappa t},
\]
this yields \(\psi\)-irreducibility on \((\bar\lambda,\infty)\), aperiodicity, and petite compact sets.

The drift inequality above then gives positive Harris recurrence and a unique invariant probability \(\pi_\lambda\) by \cite[Theorem~4.2]{meyn1993stability}. Combined with the preceding irreducibility, aperiodicity, and petite-set properties, it also yields \(\mathcal W\)-uniform geometric ergodicity by \cite[Theorems~5.1--5.2]{down1995exponential}. Therefore, for some \(\mathfrak M<\infty\) and \(\mathfrak C>0\),
\[
\left\|\mathcal L(\lambda_t\mid\lambda_0)-\pi_\lambda\right\|_{\mathrm{TV}}\le\mathfrak M\mathcal W(\lambda_0)e^{-\mathfrak C t}=\mathfrak M(1+\lambda_0)e^{-\mathfrak C t},
\qquad t\ge0 .
\]

It remains to identify the stationary mean. Applying the generator to
\(V(\lambda)=\lambda\) gives
\[
\frac{d}{dt}\mathbb E[\lambda_t\mid\lambda_0]=\kappa\bar\lambda-(\kappa-\eta\bar g)\mathbb E[\lambda_t\mid\lambda_0].
\]
Therefore
\[
\mathbb E[\lambda_t\mid\lambda_0]=e^{-(\kappa-\eta\bar g)t}\lambda_0+\frac{\kappa\bar\lambda}{\kappa-\eta\bar g}\left(1-e^{-(\kappa-\eta\bar g)t}\right),
\]
which is precisely \eqref{eq:first_moment_m}. Since
\(\mathcal W\)-uniform ergodicity implies
\[
\int_0^\infty \mathcal W(\lambda)\,\pi_\lambda(d\lambda)<\infty ,
\]
the invariant distribution has finite first moment. Taking
\(\lambda_0\sim\pi_\lambda\) in the preceding mean identity and using
stationarity yields, for every \(t>0\),
\[
\int_0^\infty \lambda\,\pi_\lambda(d\lambda)=e^{-(\kappa-\eta\bar g)t}\int_0^\infty \lambda\,\pi_\lambda(d\lambda)+\frac{\kappa\bar\lambda}{\kappa-\eta\bar g}\left(1-e^{-(\kappa-\eta\bar g)t}\right).
\]
Since \(1-e^{-(\kappa-\eta\bar g)t}>0\), it follows that
\[
\int_0^\infty \lambda\,\pi_\lambda(d\lambda)=\frac{\kappa\bar\lambda}{\kappa-\eta\bar g}.
\]
This coincides with the limiting value of the finite-horizon mean in
\eqref{eq:first_moment_m}.

\section{}\label{app:proof_real_axis_riccati}
\textbf{Proof of Theorem~\ref{thm:real_axis_riccati}:} For fixed \(u\in\mathbb R\) and \(z\in\mathbb C_-\), set
\[
\mathcal J_u(z):=\int_{\mathbb R_0}\left[\exp\left(iuy+\eta z g(y)\right)-1-iuy\mathbf 1_{\{|y|<1\}}\right]\bar\nu(dy;\vartheta),
\]
whenever the integral is finite. We first verify that this integral is indeed well defined on \(\mathbb C_-\).

At \(z=0\), the finiteness follows by the same small- and large-jump argument as in Lemma~\ref{lem:jump_martingale_correction_finite}. Now let \(z_1,z_2\in\mathbb C_-\). Since \(\mathbb C_-\) is convex, the line segment joining \(z_1\) and \(z_2\) remains in \(\mathbb C_-\). Using the integral mean-value identity for the complex exponential, we obtain
\[
\begin{aligned}
\left| e^{\eta z_1g(y)} - e^{\eta z_2g(y)} \right| &\le \eta |z_1-z_2|g(y) \sup_{\theta\in[0,1]} \left| \exp\left( \eta\bigl(\theta z_1+(1-\theta)z_2\bigr)g(y) \right) \right| \\ &\le \eta |z_1-z_2|g(y). \end{aligned}
\]
Therefore,
\[
|\mathcal J_u(z_1)-\mathcal J_u(z_2)| \le \eta\bar g\,|z_1-z_2|.
\]
In particular, taking \(z_2=0\) shows that \(\mathcal J_u(z)\) is well defined for every \(z\in\mathbb C_-\). By \eqref{eq:K_1},
we obtain
\[
|\mathcal K_1(u,z_1)-\mathcal K_1(u,z_2)| \le (\kappa+\eta\bar g)|z_1-z_2|, \qquad z_1,z_2\in\mathbb C_-.
\]
Moreover,
\[
|\mathcal K_1(u,z)| \le |\mathcal K_1(u,0)| + (\kappa+\eta\bar g)|z| \le C_0+C_1|z|,
\qquad z\in\mathbb C_-
\]
for finite constants \(C_0, C_1\) depending on \(u\) and the model parameters.

Since \(\mathcal K_1(u,\cdot)\) is Lipschitz on \(\mathbb C_-\), the Riccati equation in \eqref{eq:riccati_psi}, viewed as a two-dimensional real ordinary differential equation for \((\operatorname{Re}\psi,\operatorname{Im}\psi)\), admits a unique local solution by the standard ODE existence and uniqueness theorem \cite[Chapter~1]{coddington1955theory}. The linear-growth bound will be used below, together with the forward invariance of \(\mathbb C_-\), to rule out finite-time explosion. We write
\[
\psi(\tau;u,v)=\xi(\tau)+i\zeta(\tau),
\qquad
\xi(\tau)=\operatorname{Re}\psi(\tau;u,v).
\]
Taking real parts in the Riccati equation \eqref{eq:generalized_Riccati_system}, and noting that
\(-iu\chi_J\) and
\(-iuy\mathbf 1_{\{|y|<1\}}\) are purely imaginary, gives
\[
\xi'(\tau)
=
-\kappa\xi(\tau)
+
\int_{\mathbb R_0}
\left[
e^{\eta \xi(\tau)g(y)}
\cos\left(uy+\eta\zeta(\tau)g(y)\right)
-
1
\right]
\bar\nu(dy;\vartheta).
\]
We now verify the sub-tangentiality condition on the boundary of the closed
half-plane
\[
\mathbb C_-=\{z\in\mathbb C:\operatorname{Re}z\le0\}.
\]
On the boundary \(\partial\mathbb C_-\), where \(\xi(\tau)=0\), the real-part
equation reduces to
\[
\xi'(\tau)
=
\int_{\mathbb R_0}
\left[
\cos\left(uy+\eta\zeta(\tau)g(y)\right)-1
\right]
\bar\nu(dy;\vartheta)
\le0,
\]
because the cosine term is bounded above by one.
Equivalently, at each boundary point of \(\mathbb C_-\), the Riccati vector
field belongs to the tangent cone of \(\mathbb C_-\). Thus the vector field
satisfies the Nagumo sub-tangentiality condition on \(\partial\mathbb C_-\).
By the tangent-cone form of Nagumo's positive-invariance criterion \cite[Theorem~3.1]{blanchini1999set}, the closed half-plane \(\mathbb C_-\) is
forward invariant for the Riccati equation. Since the initial value satisfies
\[
v\in\mathbb C_-,
\]
we conclude that
\[
\psi(\tau;u,v)\in\mathbb C_-
\]
for all times at which the Riccati solution exists.

It remains to exclude finite-time explosion of the Riccati solution. Fix an
arbitrary finite horizon \(\bar T<\infty\). On every interval on which the local
solution is defined, the forward-invariance argument above gives
\[
\psi(\tau;u,v)\in\mathbb C_-.
\]
Hence the linear-growth estimate for the Riccati vector field on \(\mathbb C_-\)
implies
\[
|\psi'(\tau;u,v)|
\le
C_0+C_1|\psi(\tau;u,v)|.
\]
Consequently, for all such \(\tau\le \bar T\),
\[
|\psi(\tau;u,v)|
\le
|v|
+
\int_0^\tau
\left(
C_0+C_1|\psi(s;u,v)|
\right)ds.
\]
By Gronwall's inequality,
\[
|\psi(\tau;u,v)|
\le
\left(
|v|+\frac{C_0}{C_1}
\right)e^{C_1\tau}
-
\frac{C_0}{C_1},
\qquad 0\le \tau\le \bar T
\]
with the evident modification when \(C_1=0\). Thus \(\psi\) remains bounded on every finite horizon. The standard continuation theorem for ODEs then rules out finite-time explosion. Since \(\bar T<\infty\) was arbitrary, the Riccati solution exists globally on every finite horizon.

Now define
\[
\phi(\tau;u,v)=\int_0^\tau\mathcal K_0\bigl(u,\psi(s;u,v)\bigr)\,ds.
\]
Since
\[
\mathcal K_0(u,z)=iu\left(r-\frac12\sigma^2\right)-\frac12\sigma^2u^2+\kappa\bar\lambda z
\]
is affine in \(z\), and \(\psi(\cdot;u,v)\) is finite on finite horizons, the integral defining \(\phi\) is finite on every finite horizon. Hence
\[
\phi'(\tau;u,v)=\mathcal K_0\bigl(u,\psi(\tau;u,v)\bigr),
\qquad
\phi(0;u,v)=0.
\]

For \(s\in[t,T]\), define
\[
\mathcal M_s:=f_{T-s}(X_s,\lambda_s)=\exp\left(iuX_s+\phi(T-s;u,v)+\psi(T-s;u,v)\lambda_s\right).
\]
Since \(u\in\mathbb R\), \(\lambda_s\ge0\) by Proposition~\ref{prop:well_posed_activity} and
\[
\operatorname{Re}\psi(T-s;u,v)\le0,
\]
we have
\[
|\mathcal M_s|\le\exp\left(\sup_{0\le q\le T-t}\operatorname{Re}\phi(q;u,v)\right)<\infty.
\]
Thus \(\mathcal M\) is bounded on \([t,T]\).

We now identify its drift. Since \(q=T-s\), the Riccati equations give
\[
\frac{\partial_s f_{T-s}(x,\ell)}{f_{T-s}(x,\ell)}
=
-\mathcal K_0\bigl(u,\psi(T-s;u,v)\bigr)
-
\ell\,\mathcal K_1\bigl(u,\psi(T-s;u,v)\bigr).
\]
On the other hand, the generator calculation yields
\[
\frac{\mathcal A f_{T-s}(x,\ell)}{f_{T-s}(x,\ell)}
=
\mathcal K_0\bigl(u,\psi(T-s;u,v)\bigr)
+
\ell\,\mathcal K_1\bigl(u,\psi(T-s;u,v)\bigr).
\]
Therefore,
\[
\left(\partial_s+\mathcal A\right)f_{T-s}(x,\ell)=0.
\]

Applying It\^o's formula for jump semimartingales to
\(f_{T-s}(X_s,\lambda_s)\) shows that \(\mathcal M\) is a local martingale.
Since \(\mathcal M\) is bounded on \([t,T]\), it is a true martingale. Hence
\[
\mathcal M_t
=
\mathbb E_t[\mathcal M_T].
\]
At \(s=T\), we have
\[
\phi(0;u,v)=0,
\qquad
\psi(0;u,v)=v,
\]
and therefore
\[
\mathcal M_T
=
\exp\left(iuX_T+v\lambda_T\right).
\]
Consequently,
\[
\mathbb E_t
\left[
\exp\left(iuX_T+v\lambda_T\right)
\right]
=
\exp\left(
iuX_t+\phi(T-t;u,v)+\psi(T-t;u,v)\lambda_t
\right).
\]

\section{}\label{app:true_martingale_sufficient_condition}
\textbf{Proof of Theorem~\ref{thm:true_martingale_sufficient_condition}:} The finiteness of \(H_J<\infty\) follows by the same estimates used for Lemma~\ref{lem:jump_martingale_correction_finite}: near zero the integrand is \(O(y^2)\), while the positive tail is controlled by the exponential tempering under \(M>1\), and the negative tail is controlled by the L\'evy-measure property.

By Proposition~\ref{prop:risk_neutral_drift_restriction}, the discounted stock price is a nonnegative local martingale. Equivalently, the normalized discounted stock price admits the stochastic-exponential representation
\[
\frac{e^{-rt}S_t}{S_0}
=
\mathcal E(\mathcal Y)_t,
\]
where \(\mathcal Y\) is the stochastic logarithm and decomposes as
\[
\mathcal Y_t=\mathcal Y_t^c+\mathcal Y_t^d.
\]
The continuous local-martingale part is \(\mathcal Y_t^c=\sigma W_t\), and the purely discontinuous local-martingale part is
\[
\mathcal Y_t^d=\int_0^t\int_{\mathbb R_0}\left(e^y-1\right)\left(\mu(ds,dy)-\lambda_{s-}\bar\nu(dy;\vartheta)\,ds\right).
\]
This representation is the standard semimartingale form in which exponential-martingale criteria are applied; see \cite[Section~2]{kallsen2002cumulant} for the stochastic-logarithm and exponential-compensator framework.

At a jump time of \(\mathcal Y\) with log-price jump size \(y\), we have
\[
\Delta \mathcal Y=e^y-1.
\]
Therefore,
\[
(1+\Delta\mathcal Y)\log(1+\Delta\mathcal Y)-\Delta\mathcal Y=e^y y-e^y+1.
\]
This is the jump-entropy term appearing in the L\'epingle--M\'emin exponential-martingale criterion \cite[Th\'eor\`eme~III.1]{lepingle1978integrabilite}. We apply this criterion to the stopped local martingale
\[
\mathcal Y^T_s:=\mathcal Y_{s\wedge T},
\]
which reduces to a finite-horizon condition. For the Dol\'eans exponential \(\mathcal E(\mathcal Y^T)\), the L\'epingle--M\'emin criterion involves the predictable compensator of
\[
\frac12\langle\mathcal Y^c,\mathcal Y^c\rangle_t+\sum_{s\le t}\left[(1+\Delta\mathcal Y_s)\log(1+\Delta\mathcal Y_s)-\Delta\mathcal Y_s\right],
\qquad 0\le t\le T.
\]
Since
\[
\mathcal Y_t^c=\sigma W_t,
\qquad
\langle\mathcal Y^c,\mathcal Y^c\rangle_t=\sigma^2t,
\]
the continuous contribution on \([0,T]\) is
\[
\frac12\langle\mathcal Y^c,\mathcal Y^c\rangle_T=\frac12\sigma^2T.
\]
Moreover, the predictable compensator of the jump contribution is
\[
\int_0^T
\lambda_s
\int_{\mathbb R_0}
\left(
y e^y-e^y+1
\right)
\bar\nu(dy;\vartheta)\,ds
=
H_J\int_0^T\lambda_s\,ds.
\]
Hence the L\'epingle--M\'emin criterion applies if
\[
\mathbb E\left[\exp\left(\frac12\sigma^2T+H_J\int_0^T\lambda_s\,ds\right)\right]<\infty.
\]
Since \(e^{\sigma^2T/2}\) is a finite deterministic factor, this condition follows from
\[
\mathbb E\left[\exp\left(H_J\int_0^T\lambda_s\,ds\right)\right]<\infty.
\]
Consequently, \(\mathcal E(\mathcal Y)\) is a true martingale on \([0,T]\). Since
\[
\frac{e^{-rt}S_t}{S_0}=\mathcal E(\mathcal Y)_t,
\]
we obtain, for every \(0\le t\le T\),
\[
\mathbb E_t[e^{-rT}S_T]=e^{-rt}S_t.
\]
Equivalently,
\[
\mathbb E_t[S_T]=e^{r(T-t)}S_t.
\]
Thus the discounted stock price \(e^{-rt}S_t\) is a true martingale on finite horizons.

\section{}\label{app:cos_cumulants}
\textbf{Conditional Cumulants for the COS Truncation Rule:} The conditional cumulants of \(Z_T\) are recovered from the local expansion of the log-characteristic function at the origin:
\[
\log\varphi_t(u;T,K)=\sum_{n=1}^4\frac{(iu)^n}{n!}\mathfrak c_n^Z+
o(u^4),
\qquad u\to0.
\]

By \eqref{eq:characteristic_function_of_payoff} and \eqref{eq:characteristic_function_of_log-moneyness}, we obtain the local log-characteristic function
\[
\log\varphi_t(u;T,K)=iu(X_t-\log K)+\phi(\tau;u,0)+\psi(\tau;u,0)\lambda_t,
\]
where the logarithm is understood locally around \(u=0\). Define the real-axis Riccati derivatives, for the finite orders considered below, by
\[
a_n(s):=\frac{1}{i^n}\left.\frac{\partial^n}{\partial u^n}\phi(s;u,0)\right|_{u=0},
\qquad
b_n(s):=\frac{1}{i^n}\left.\frac{\partial^n}{\partial u^n}\psi(s;u,0)\right|_{u=0}.
\]
Then, for \(\tau=T-t\),
\[
\mathfrak c_n^Z=\mathbf 1_{\{n=1\}}(X_t-\log K)+a_n(\tau)+\lambda_t b_n(\tau),
\qquad n\geq1.
\]

We now derive the derivative system for \(a_n\) and \(b_n\). For \(v=0\), the Riccati equations are
\[
\partial_s\phi(s;u,0)=iu\left(r-\frac12\sigma^2\right)-\frac12\sigma^2u^2+\kappa\bar\lambda\,\psi(s;u,0)
\]
and
\begin{align*}
\partial_s\psi(s;u,0)=
&-\kappa\psi(s;u,0)-iu\chi_J\\
&+\int_{\mathbb R_0}\left(e^{iuy+\eta\psi(s;u,0)g(y)}-1-iuy\mathbf 1_{\{|y|<1\}}\right)\bar\nu(dy;\vartheta)
\end{align*}
with
\[
\phi(0;u,0)=0,
\qquad
\psi(0;u,0)=0.
\]
Differentiating the \(\phi\)-equation at \(u=0\) gives
\[
a_1'(s)=r-\frac12\sigma^2+\kappa\bar\lambda\,b_1(s),
\]
\[
a_2'(s)=\sigma^2+\kappa\bar\lambda\,b_2(s),
\]
and, for \(n\ge3\),
\[
a_n'(s)=\kappa\bar\lambda\,b_n(s)
\]
with \(a_n(0)=0\).

It remains to compute the corresponding derivatives of \(\psi\). The first derivative satisfies
\[
b_1'(s)=-(\kappa-\eta\bar g)\,b_1(s)+\int_{|y|\ge1}y\,\bar\nu(dy;\vartheta)-\chi_J, \qquad b_1(0)=0.
\]
Then we obtain
\[
b_1(s)=\frac{\displaystyle\int_{|y|\ge1}y\,\bar\nu(dy;\vartheta)-\chi_J}{\kappa-\eta\bar g}\left(1-e^{-(\kappa-\eta\bar g)\cdot s}\right).
\]
Therefore
\[
\mathfrak c_1^Z=X_t-\log K+a_1(\tau)+\lambda_t b_1(\tau),
\]
where
\[
a_1(\tau)=\left(r-\frac12\sigma^2\right)\tau+\kappa\bar\lambda\int_0^\tau b_1(s)\,ds.
\]

For higher-order derivatives, set
\[
\tilde{y}=\tilde{y}(y,s):=y+\eta b_1(s)g(y).
\]
The second derivative satisfies
\[
b_2'(s)=-(\kappa-\eta\bar g)\,b_2(s)+\int_{\mathbb R_0}\tilde{y}^2\bar\nu(dy;\vartheta),
\qquad b_2(0)=0.
\]
Consequently,
\[
\mathfrak c_2^Z=a_2(\tau)+\lambda_t b_2(\tau),
\]
where
\[
a_2(\tau)=\sigma^2\tau+\kappa\bar\lambda\int_0^\tau b_2(s)\,ds.
\]

The fourth cumulant depends on the third Riccati derivative, so we also record the equation for \(b_3\):
\[
b_3'(s)=-(\kappa-\eta\bar g)\,b_3(s)+\int_{\mathbb R_0}\left[\tilde{y}^3+3\eta\,\tilde{y}\,b_2(s)\,g(y)\right]\bar\nu(dy;\vartheta)
\]
with \(b_3(0)=0\). Then the fourth derivative satisfies
\[
\begin{aligned}
b_4'(s)
&=-(\kappa-\eta\bar g)\,b_4(s)+\int_{\mathbb R_0}\Big[\tilde{y}^4+6\eta\,\tilde{y}^2\,b_2(s)\,g(y)\\
&+3\big(\eta\,b_2(s)\,g(y)\big)^2+4\eta\,\tilde{y}\,b_3(s)\,g(y)\Big]\bar\nu(dy;\vartheta)
\end{aligned}
\]
with \(b_4(0)=0\). Therefore
\[
\mathfrak c_4^Z
=
a_4(\tau)+\lambda_t b_4(\tau),
\qquad
a_4(\tau)
=
\kappa\bar\lambda\int_0^\tau b_4(s)\,ds.
\]

Thus, the cumulants required by the COS truncation rule, \[ \mathfrak c_1^Z, \qquad \mathfrak c_2^Z, \qquad \mathfrak c_4^Z, \] are obtained from the real-axis Riccati derivative system above. The required jump-size and mixed jump-feedback integrals are finite under Theorem~\ref{thm:normalized-shape-admissibility}, the bounded-feedback specification \(g(y)=1-e^{-ay^2}\), and the finite-horizon admissibility of the activity scale. This computation uses only derivatives at \(u=0\) along the real Fourier axis and does not impose an additional complex-transform admissibility condition.

\section{}\label{app:Carr_Madan}
\textbf{Proof of Proposition~\ref{prop:fourier_call_pricing}:} By \eqref{eq:European_valuation}, the European call price is
\[
C_t(K,T)=e^{-r\tau}\mathbb E_t\left[(e^{X_T}-K)^+\right].
\]
For \(\delta>0\), define the damped call-price function as a function of the log-strike variable \(h\in\mathbb R\) by
\[
c_\delta(h):=e^{\delta h}C_t(e^h,T).
\]
We use the Fourier-transform convention
\[
\widehat c_\delta(\omega):=\int_{\mathbb R}e^{i\omega h}c_\delta(h)\,dh,
\qquad
c_\delta(h)=\frac{1}{2\pi}\int_{\mathbb R}e^{-i\omega h}\widehat c_\delta(\omega)\,d\omega .
\]
Under Assumption~\ref{ass:carr_madan_benchmark_admissibility}, the damped call-price function and the corresponding Fourier integrand are integrable, so that Fubini's theorem and Fourier inversion are applicable. Hence
\[
\begin{aligned}
\widehat c_\delta(\omega)
&=\int_{\mathbb R}e^{i\omega h}e^{\delta h}e^{-r\tau}\mathbb E_t\left[(e^{X_T}-e^h)^+\right]dh  \\
&=e^{-r\tau}\mathbb E_t\left[\int_{\mathbb R}e^{(i\omega+\delta)h}(e^{X_T}-e^h)^+\,dh\right].
\end{aligned}
\]
For a fixed realization \(X_T=x\), the payoff is positive only for \(h<x\).
Therefore,
\[
\begin{aligned}
\int_{\mathbb R}e^{(i\omega+\delta)h}(e^x-e^h)^+\,dh
&=\int_{-\infty}^{x}e^{(i\omega+\delta)h}(e^x-e^h)\,dh\\
&=e^x\int_{-\infty}^{x}e^{(i\omega+\delta)h}\,dh-\int_{-\infty}^{x}e^{(i\omega+\delta+1)h}\,dh.
\end{aligned}
\]
Since \(\delta>0\), both integrals are finite at \(-\infty\). Direct calculation gives
\[
\begin{aligned}
\int_{-\infty}^{x}e^{(i\omega+\delta)h}(e^x-e^h)\,dh
&=e^x\frac{e^{(i\omega+\delta)x}}{\delta+i\omega}-\frac{e^{(i\omega+\delta+1)x}}{\delta+1+i\omega} \\
&=e^{(i\omega+\delta+1)x}\left(\frac{1}{\delta+i\omega}-\frac{1}{\delta+1+i\omega}\right) \\
&=\frac{e^{(i\omega+\delta+1)x}}{(\delta+i\omega)(\delta+1+i\omega)
}.
\end{aligned}
\]
Applying this identity with \(x=X_T\), we obtain
\[
\widehat c_\delta(\omega)=e^{-r\tau}\frac{\mathbb E_t\left[e^{(i\omega+\delta+1)X_T}\right]}{(\delta+i\omega)(\delta+1+i\omega)}.
\]
Since
\[
e^{(i\omega+\delta+1)X_T}=e^{i(\omega-i(\delta+1))X_T},
\]
the numerator is the complex extension of the conditional log-price transform evaluated at
\[
u=\omega-i(\delta+1).
\]
Thus
\[
\mathbb E_t\left[e^{(i\omega+\delta+1)X_T}\right]=\Phi_t\bigl(\omega-i(\delta+1);T\bigr).
\]
Moreover,
\[
(\delta+i\omega)(\delta+1+i\omega)=\delta^2+\delta-\omega^2+i(2\delta+1)\omega .
\]
Hence
\[
\widehat c_\delta(\omega)=e^{-r\tau}\frac{\Phi_t\bigl(\omega-i(\delta+1);T\bigr)}{\delta^2+\delta-\omega^2+i(2\delta+1)\omega}.
\]

By Fourier inversion,
\[
C_t(e^h,T)=e^{-\delta h}c_\delta(h)=\frac{e^{-\delta h}}{2\pi}\int_{-\infty}^{\infty}e^{-i\omega h}\widehat c_\delta(\omega)\,d\omega.
\]
Since \(c_\delta(h)\) is real-valued, its Fourier transform satisfies
\[
\widehat c_\delta(-\omega)=\overline{\widehat c_\delta(\omega)}.
\]
Thus
\[
\frac{1}{2\pi}\int_{-\infty}^{\infty}e^{-i\omega h}\widehat c_\delta(\omega)\,d\omega=\frac{1}{\pi}\int_0^\infty\operatorname{Re}\left[e^{-i\omega h}\widehat c_\delta(\omega)\right]d\omega .
\]
Substituting the expression for \(\widehat c_\delta(\omega)\) yields
\[
C_t(e^h,T)=\frac{e^{-\delta h-r\tau}}{\pi}\int_0^\infty\operatorname{Re}\left[e^{-i\omega h}\frac{\Phi_t\bigl(\omega-i(\delta+1);T\bigr)}{\delta^2+\delta-\omega^2+i(2\delta+1)\omega}\right]d\omega.
\]
Evaluating this identity at \(h=\log K\), and using
\[
e^{-\delta\log K}=K^{-\delta},
\qquad
e^{-i\omega\log K}=K^{-i\omega},
\]
we obtain
\[
C_t(K,T)=\frac{K^{-\delta}e^{-r\tau}}{\pi}\int_0^\infty\operatorname{Re}\left[K^{-i\omega}\frac{\Phi_t\bigl(\omega-i(\delta+1);T\bigr)}{\delta^2+\delta-\omega^2+i(2\delta+1)\omega}\right]d\omega ,
\]
which proves the claim.

\section{}\label{app:eta_zero_benchmark}
\textbf{No-Feedback Benchmark:} The Riccati equation for \(\psi\) reduces to the linear equation
\[
\psi'(\tau;u,0)=-\kappa\psi(\tau;u,0)-iu\chi_J+\int_{\mathbb R_0}\left[e^{iuy}-1-iuy\mathbf 1_{\{|y|<1\}}\right]\bar\nu(dy;\vartheta).
\]
Hence, we obtain
\[
\psi(\tau;u,0)=\frac{\mathcal K_1(u,0)}{\kappa}\left(1-e^{-\kappa\tau}\right).
\]
The corresponding \(\phi\)-equation is
\[
\phi'(\tau;u,0)=iu\left(r-\frac12\sigma^2\right)-\frac12\sigma^2u^2+\kappa\bar\lambda\,\psi(\tau;u,0).
\]
Therefore,
\[
\phi(\tau;u,0)=\left[iu\left(r-\frac12\sigma^2\right)-\frac12\sigma^2u^2\right]\tau+\bar\lambda \mathcal K_1(u,0)\left[\tau-\frac{1-e^{-\kappa\tau}}{\kappa}\right].
\]
The no-feedback characteristic function is
\[
\Phi_t^{(0)}(u;T)=\exp\left(iuX_t+\phi(\tau;u,0)+\psi(\tau;u,0)\lambda_t\right).
\]
This benchmark isolates the role of the endogenous feedback parameter \(\eta\). When \(\eta=0\), the activity state remains mean-reverting, but it is not excited by realized log-price jumps.
\end{appendices}

\bigskip

\noindent\textbf{Data Availability}

The numerical algorithms and source code that support the findings of this study are available from the corresponding author upon reasonable request.

\bigskip
\noindent\textbf{Acknowledgements}

This work was supported by  Guangdong Basic and Applied Basic Research Foundation (Grant No. 2025A1515012560),  Guangdong Introduction Program (Grant No. 2023QN10X753) and  National Foreign Experts Program (Grant No. 111001819820258003).

\bibliography{ref}

\end{document}